\documentclass[11pt,reqno]{amsart}

\usepackage{amsmath, amsthm, amscd, amsfonts, amssymb, graphicx, color, mathrsfs}
\usepackage{mathscinet}

\usepackage{eurosym}
\usepackage[utf8]{inputenc}
\usepackage{amsmath}
\usepackage{amsfonts}
\usepackage{version}
\usepackage{mathtools}
\usepackage{amscd}
\usepackage{amssymb}
\usepackage{dsfont}
\usepackage{amsthm}
\usepackage{thmtools}
\usepackage{graphicx}
\usepackage{mathrsfs}
\usepackage{todonotes}
\usepackage{xcolor}
\usepackage{todonotes}
\usepackage{setspace}
\usepackage{enumerate}
\usepackage[english]{babel}
\usepackage{xcolor}
\usepackage[colorlinks=true,linkcolor=blue]{hyperref}
\usepackage[all]{hypcap}

\renewcommand{\epsilon}{\varepsilon}
\renewcommand{\rho}{\varrho}
\renewcommand{\phi}{\varphi}
\renewcommand{\d}{{\rm d}}
\newcommand{\eins}{\mathds{1}}

\newcommand{\E}{\mathbb{E}}
\newcommand{\N}{\mathbb{N}}
\renewcommand{\P}{\mathbb{P}}

\newcommand{\R}{\mathbb{R}}

\newcommand{\Bc}{\mathcal{B}}

\newcommand{\Pc}{\mathcal{P}}

\newcommand{\Wc}{\mathcal{W}}

\newcommand{\Cpl}{\mathop{\rm Cpl}}
\newcommand{\dom}{\mathop{\rm dom}}
\newcommand{\Lip}{{\mathop{\rm Lip}\nolimits}}
\newcommand{\Lipb}{\Lip_{\rm b}}

\newcommand{\Con}{{{\rm C}}}
\newcommand{\Cb}{{\Con_{\rm b}}}
\newcommand{\BUC}{\Con_{\rm b}}

\newcommand{\ep}{\epsilon}

\newcommand{\la}{\lambda}

\newtheorem{theorem}{Theorem}[section]
\newtheorem{lemma}[theorem]{Lemma}
\newtheorem{proposition}[theorem]{Proposition}

\theoremstyle{definition}
\newtheorem{definition}[theorem]{Definition}
\newtheorem{example}[theorem]{Example}

\newtheorem{notation}[theorem]{Notation}
\newtheorem{remark}[theorem]{Remark}
\newtheorem{assumption}[theorem]{Assumption}

\numberwithin{equation}{section}

\begin{document}
	
	%%%%%%%%%%%%%%%%%%%%%%%%%%%%%%%%%%%%%%%%%%%%%%%%%%%%%%%%%%%%%%%%%%%%%%%%%%%%%%%%%%%%%%%%%%%%%%%%%%%%%%%%%%%%%%%%%%%%%%%%

	\title[Wasserstein Perturbation of Markovian semigroups]{Wasserstein perturbations of Markovian\\ transition semigroups}

	\author{Sven Fuhrmann}
	\address{Department of Mathematics and Statistics, University of Konstanz}
	\email{sven.fuhrmann@uni-konstanz.de}

	\author{Michael Kupper}
	\address{Department of Mathematics and Statistics, University of Konstanz}
	\email{kupper@uni-konstanz.de}

	\author{Max Nendel}
	\address{Center for Mathematical Economics, Bielefeld University}
	\email{max.nendel@uni-bielefeld.de}

	\thanks{We thank Daniel Bartl, Jonas Blessing, and Stephan Eckstein for valuable comments and discussions, as well as two anonymous referees for their helpful suggestions and remarks. Financial support through the Deutsche Forschungsgemeinschaft (DFG, German Research Foundation) -- SFB 1283/2 2021 -- 317210226 is gratefully acknowledged.}
	
	\date{\today}

	\begin{abstract}
		In this paper, we deal with a class of time-homogeneous continuous-time Markov processes with transition probabilities bearing a nonparametric uncertainty. The uncertainty is modelled by considering perturbations of the transition probabilities within a proximity in Wasserstein distance. As a limit over progressively finer time periods, on which the level of uncertainty scales proportionally, we obtain a convex semigroup satisfying a nonlinear PDE in a viscosity sense.\
		A remarkable observation is that, in standard situations, the nonlinear transition operators arising from nonparametric uncertainty coincide with the ones related to parametric drift uncertainty.  On the level of the generator, the uncertainty is reflected as an additive perturbation in terms of a convex functional of first order derivatives. We additionally provide sensitivity bounds for the convex semigroup relative to the reference model. The results are illustrated with Wasserstein perturbations of L\'evy processes, infinite-dimensional Ornstein-Uhlenbeck processes, geometric Brownian motions, and Koopman semigroups.\\
		
		\smallskip
		
			
	\end{abstract}
	
	\maketitle
	
	\setcounter{tocdepth}{1}
	\section{Introduction}
	When considering stochastic processes for the modeling of  real world phenomena, a major issue is so-called \textit{model uncertainty} or \textit{epistemic uncertainty}.\ The latter refer to the impossibility of perfectly capturing information about the future in a single stochastic framework.\ The relevance of model uncertainty, in a stochastic and nonstochastic setting, is widely recognised within many fields such as statistics, operations research, finance, and economic theory, and one typically differentiates between parametric and nonparametric uncertainty.
	
	Parametric uncertainty relates to the lack of information regarding certain parameters of a model while taking other model-specific assumptions as given. In a dynamic setting, the construction of consistent families of nonlinear transition semigroups related to parameter uncertainty has recently received a lot of attention, cf. Coquet et al.\ \cite{MR1906435}, Denk et al.\ \cite{denk2020semigroup},  Fadina et al.\ \cite{MR3955321}, Hu and Peng \cite{hu2009glevy}, K\"uhn \cite{kuhn2018viscosity}, Neufeld and Nutz \cite{neufeld2017nonlinear},  and Peng \cite{PengG}. Nonlinear semigroups are intimately related to BSDEs and stochastic optimal control, cf.\ Cheridito et al.\ \cite{MR2319056}, El Karoui et al.\ \cite{MR1434407}, Kazi-Tani et al.\ \cite{MR3361253}, and Soner et al.\ \cite{MR2925572}.
	
	On the other hand,  nonparametric uncertainty refers to the impossibility of precisely capturing certain aspects of a model, such as independence or, more generally, the joint distributions of certain random variables -- phenomena that are in general not implementable using a finite number of parameters. 
    One way to tackle such situations is to consider perturbations of a reference model using a certain notion of proximity, e.g., in terms of Wasserstein distances, cf.\ Bartl et al.\ \cite{bartl2020computational}, Blanchet and Murthy \cite{blanchet2019quantifying}, Mohajerin Esfahani and Kuhn \cite{esfahani2018data}, Gao and Kleywegt \cite{gao2016distributionally}, Pflug and Wozabal \cite{pflug2007ambiguity}, Zhao and Guan \cite{MR3771309}, as well as Bartl et al.\ \cite{bartl2020limits} for a dynamic setting.
    
     The starting point of the present article is the work of Bartl et al.\ \cite{bartl2020limits}, where perturbations of certain L\'evy processes in $\R^d$ were considered. 
	Our focus is on more general classes of stochastic processes with values in infinite-dimensional state spaces and their transition semigroups under nonparametric uncertainty. For this purpose, we start with a reference model which is a time-homogeneous Markov process $(\Xi_t^x)$ of the form 
	\begin{equation}\label{eq.intro1} \Xi_t^x:=\psi_{t}(x) + Y_t\quad \text{for }t\geq 0\text{ and }x\in X,\quad \text{with }Y_t\sim\mu_t, \end{equation}
	taking values in a separable Banach space $X$.\ Here, $(\mu_t)_{t\geq 0}$ is a family of probability measures satisfying a certain moment condition.\
	Although at first glance slightly restrictive, this class of Markov processes is analytically tractable, can easily be simulated, and includes the following prominent dynamics, see Section \ref{sec.examples}:
	\begin{itemize}
		\item suitably integrable infinite-dimensional L\'evy processes,
		\item infinite-dimensional Ornstein-Uhlenbeck processes,
		\item geometric Brownian motions,
		\item deterministic flows arising from solutions to Banach-space-valued differential equations.
	\end{itemize}
The multiplicative structure of the geometric Brownian motion can for instance be captured by the dynamics $\Xi_t^x=\psi_{t}(x) + Y_t:=xY_t$, where the ``addition'' is understood as the multiplication on $(0,\infty)$, see Section \ref{sec.discussion} for further details. Given a Markov process of the form \eqref{eq.intro1}, we take nonparametric uncertainty of the respective transition probabilities into account by considering the worst-case scenario among perturbations within a proximity in Wasserstein distance. This is expressed by using a penalisation function $\phi\colon [0,\infty)\to [0,\infty]$ applied to the Wasserstein distance to the family $(\mu_t)_{t\geq 0}$. A canonical example would be the indicator function 
$\phi=\infty \cdot \eins_{(a,\infty)}$ with $a\geq 0$, which allows for all transition probabilities in a Wasserstein neighborhood around the reference process. Then, the worst-case expectation at time $t\geq 0$ of a suitable function $ u $ in state $x$ is given by
	\[ \big(I(t)u\big)(x):=\sup \int_X u\big(\psi_t(x)+z\big)\,\nu(\d z), \]
	where the supremum is taken over all $ \nu $ with Wasserstein distance $ \Wc(\mu_t,\nu)\leq a t$ and the level of uncertainty $a t$ is proportional to the time horizon $t$. The family $(I(t))_{t\geq 0}$ can be seen as solutions to static optimization problems, which, in general, do not fulfill the Dynamic Programming Principle or the Chapman-Kolmogorov equations. In order to come up with a dynamically consistent family (semigroup), we split the optimization on $[0,t]$ into a composition of static optimization problems $I(\pi)=I(t_1)\circ I(t_2-t_1)\circ\cdots\circ I(t-t_n)$ for a partition $\pi=\{0< t_1<\cdots<t_n<t\}$.
	A crucial step in the construction is the monotonicity $I(\pi)\ge I(\pi^\prime)$ 
	for partitions $\pi\subset \pi^\prime$ of $[0,t]$, see Lemma~\ref{lem.monotonicity}.\ This allows to define $S(t)$ as a limit of $I(\pi)$ over progressively finer partitions. We emphasise that the operator $S(t)$ does \emph{not} depend on the choice of the approximating partitions, see Remark \ref{rem:Ipi}.\
    Our main result is Theorem \ref{thm.semigroup}, which states that the family $(S(t))_{t\geq 0}$ forms a strongly continuous convex semigroup on a suitable function space. Moreover, the infinitesimal behaviour of the semigroup, described in terms of its generator $A$, relates to the generator $B$ of the reference Markov process $(\Xi_t^x)$ via
	\[
	\big(Au\big)(x)=\big(Bu\big)(x)+\phi^*\big(\|u'(x)\|\big),
	\]
	where $\phi^*$ is the convex conjugate of the penalisation function $\phi$ and $u'$ denotes the Fr\'echet derivative of the function $u$. 
	In particular, under a proper definition of proximity (e.g., for a geometric Brownian motion, we consider a logarithmic version of the Wasserstein distance),
	the uncertainty in the generator does not depend on the order of the Wasserstein distance. Our results in Section \ref{sec.wasserstein} extend and complement the work \cite{bartl2020limits} as follows. 
        While the analysis in \cite{bartl2020limits} is performed on the space $\Con_0$ endowed with the supremum norm, we are working here with locally uniform convergence on a suitable space of (linearly) bounded continuous functions. Moreover, our construction allows for state-dependent dynamics such as geometric Brownian motions or Ornstein-Uhlenbeck processes. On a technical level, we make use of a suitable compact operator $K$ in order to enforce the compactness of balls. In finite dimensions, the operator $K$ can be neglected. Yet, in infinite-dimensional situations, it enables us to treat, for example, Ornstein-Uhlenbeck processes with an unbounded generator in the drift term. A minor yet subtle issue is the a priori chosen dyadic decomposition (partition) of the time interval $[0,t]$ for the definition of the operators $(S(t))_{t\geq 0}$ in \cite{bartl2020limits}. In the present work, we also use dyadic partitions for the definition of the nonlinear semigroup. However, in a second step, we show that the latter does \emph{not} depend on this particular choice of approximating partitions.
    
	In Section \ref{sec.remarks}, we discuss possible extensions of the setup and derive sensitivity bounds for the convex semigroup $(S(t))_{t\geq 0}$ relative to
	\[
	\big(T(t)u\big)(x)=\E\big[u\big(\Xi_t^x\big)\big]=\int_X u\big(\psi_t(x)+y\big)\, \mu_t(\d y).
	\]
	Note that $(T(t))_{t\geq 0}$ is the transition semigroup corresponding to the reference Markov process. Due to our assumption on the additive separability of state and randomness, the expectation $\E[u(\Xi_t^x)]$ can be computed using a Monte-Carlo simulation for the law $\mu_t$ (independent of $x$) and simply adding the deterministic expression $\psi_t(x)$ to the simulated random variable $Y_t$, see Section \ref{sec.sensitivity} for further details. The results and possible applications are illustrated in various examples in Section \ref{sec.examples}.%, where we also depict sensitivity bounds for the resulting semigroup.
	
	The present construction is related to the one of the Nisio semigroup \cite{MR0451420}, see Section \ref{sec:param}, where the family $(I(t))_{t\ge 0}$ is replaced by a supremum over a class of linear semigroups, cf.~Denk et al.\ \cite{denk2020semigroup} and Nendel and R\"ockner \cite{roecknen}. In contrast to the Nisio semigroup, where $I(\pi)$ is increasing over refining partitions, the Wasserstein robust semigroup is obtained as a decreasing limit. Although the two approaches are based on a different notion of uncertainty (parametric vs.\ nonparametric) and a different monotonicity over refining partitions (increasing vs.\ decreasing), in standard situations, they result in the same convex semigroup $(S(t))_{t\geq 0}$ -- the Nisio semigroup related to parametric drift uncertainty. An alternative construction, which does not rely on monotonicity arguments, was recently proposed by Blessing and Kupper \cite{blessingkupper}. In the multiperiod optimization problem $I(\pi)$, we consider the Wasserstein distance on the level of transition kernels, which is related to the nested distance introduced by Pflug and Pichler \cite{pflug2012distance, pflug2014multistage}. For further results on adapted Wasserstein distances, we refer to  Backhoff-Veraguas et al.\ \cite{backhoff2019all}, Backhoff-Veraguas et al.\ \cite{backhoff2020estimating}, and the references therein. Additionally, Markov processes under Wasserstein perturbations are studied in Eckstein \cite{eckstein2019extended} and Yang \cite{yang2017convex}.
	For the broad topic of model uncertainty in Markovian systems, we refer to Wiesemann et al.\ \cite{wiesemann2013robust} and Dentcheva and Ruszczy\'{n}ski \cite{MR3807945} in the context of Markov decision processes, Rudolf and Schweizer \cite{rudolf2018perturbation} for applications to Markov chain Monte Carlo algorithms, and De Cooman \cite{MR2535022}, Hartfiel \cite{MR1725607}, Krak et al.\ \cite{MR3679158}, and \v{S}kulj \cite{skulj2009discrete} for Markov chains with interval probabilities.
	
	The paper is organized as follows. In Section \ref{sec.setup}, we introduce the setup, all standing assumptions, and basic properties. Wasserstein robust semigroups are discussed in Section \ref{sec.wasserstein}. Therein, we present the worst-case operator, as well as the iterated optimization scheme. The main result is Theorem \ref{thm.semigroup}, which provides the limiting convex semigroup and its infinitesimal generator. In Section \ref{sec.remarks}, we provide a detailed discussion on the setup and our standing assumptions, present implications for sensitivity bounds, and discuss the relation to nonlinear PDEs and parametric drift uncertainty. The paper concludes with several examples in Section \ref{sec.examples}.

	\section{Setup and basic properties}\label{sec.setup}
	Throughout, let $(X,\|\cdot\|)$ be a separable Banach space with Borel $\sigma$-algebra $\Bc(X)$. We consider a fixed compact linear operator $K\colon X\to X$, i.e., $K$ is linear and the set $\big\{Kx\colon \|x\|\leq r\}$ is compact for all $r\geq 0$. We further assume that $\{Kx\colon x\in X\}$ is dense in $X$ and that $\|Kx\|\leq \|x\|$ for all $x\in X$. Note that, for a finite-dimensional space $X$, the operator $K$ can simply be chosen to be the identity. In an infinite-dimensional setting, $K$ could be the resolvent of a generator of a compact $C_0$-semigroup on $X$. We refer to Section \ref{sec.examples} for further details. We denote by $\Lip=\Lip^K(X)$ the space of all functions $u\colon X\to \R$ satisfying
        \begin{equation}\label{eq.lipcondK}
         |u(x_1)-u(x_2)|\leq L\|K(x_1-x_2)\|\quad \text{for all }x_1,x_2\in X \text{ and some }L\geq 0.
        \end{equation}
         For $u\in \Lip$, we denote the smallest constant $L\geq 0$ such that \eqref{eq.lipcondK} is satisfied by $\|u\|_\Lip$. Moreover, let $\Lipb$ denote the subspace of all $u\in \Lip$ with $\|u\|_\infty:=\sup_{x\in X}|u(x)|<\infty$. Note that $u\in \Lipb$ if and only if there exists a bounded Lipschitz continuous function $u_0\colon X\to \R$ with $u(x)=u_0(Kx)$ for all $x\in X$.\footnote{For the nontrivial direction, define $u_0(x):=\sup_{y\in X}\big( u(y)-L\|x-Ky\|\big)$, for $x\in X$, with $L\geq \|u\|_\Lip$. The boundedness of $u_0$ follows from the density of the image $\{Kx\colon x\in X\}$.} Since $K$ is continuous, $\Lipb$ is a subspace of the space of all bounded Lipschitz continuous functions $X\to \R$.
        
        For a sequence $(u_n)_{n\in \N}$ of bounded continuous functions $X\to \R$ and $u\colon X\to \R$, we write $u_n\to u$ as $n\to \infty$ if
        \begin{equation}\label{eq.pfeil}
         \sup_{n\in \N}\|u_n\|_\infty<\infty\quad \text{and}\quad \lim_{n\to \infty}\sup_{\|x\|\leq r}|u(x)-u_n(x)|=0\quad \text{for all }r\geq 0.
        \end{equation}
        By $\BUC=\BUC^K(X)$, we denote the space of all $u\colon X\to \R$, for which there exists a sequence $(u_n)_{n\in \N}\subset \Lipb$ with $u_n\to u$ as $n\to \infty$ in the sense of \eqref{eq.pfeil}.\footnote{For a sequence $(u_n)_{n\in \N}\subset\BUC$ with $u_n\to u$ as $n\to \infty$, it follows that $u\in \BUC$. This can be seen by approximating $u_n$ with the inf-convolution $u_{n,k}\colon X\to \R,\; x\mapsto \inf_{y\in X}\big(u_n(y)+k\|Kx-Ky\|\big)$. } By construction, $\BUC$ is a subspace of the space of all bounded functions $u\colon X\to \R$, which are uniformly continuous on bounded sets.
	
	Let $p\in (1,\infty)$, and consider the set $\Pc_p(X)$ of all probability measures $\mu$ on $\Bc(X)$ with finite $p$-th moment $\int_X \|y\|^p\, \mu (\d y)<\infty$. In a similar way, we consider $\Pc_p(X\times X)$ with $X$ being replaced by $X\times X$ together with the norm $\|(x_1,x_2)\|:=\|x_1\|+\|x_2\|$ for $x_1,x_2\in X$. We endow $\Pc_p(X)$ with the $p$-Wasserstein distance $\Wc_p$, which is defined as
	\begin{equation*}
		\Wc_{p}(\mu,\nu):= \bigg(\inf_{\pi \in \Cpl(\mu ,\nu)}\int_{X\times X}\|y-z\|^{p}\,\pi(\d y,\d z)\bigg)^{\tfrac1p}\quad \mbox{for }\mu,\nu \in \Pc_p(X),
	\end{equation*}
	where $\Cpl(\mu,\nu)$ denotes the set of all probability measures on $\Bc(X\times X)$ with first marginal $ \mu $ and second marginal $ \nu $, respectively. Then, $\Wc_p$ defines a metric on $\Pc_p(X)$ and, in a similar way, the Wasserstein metric $\Wc_p$ can be defined on $\Pc_p(X\times X)$. Throughout, we endow $\Pc_p(X\times X)$ with the topology generated by $\Wc_p$.
	
	In the following remark, we collect some basic properties related to the Wasserstein distance $\Wc_p$ that we frequently use in this work.\ For a detailed discussion on transport distances and their properties, we refer to Ambrosio et al.\ \cite{MR2401600} or Villani \cite{villani2008optimal}.
	
	\begin{remark}\label{rem.Wp}\
	 \begin{enumerate}
	  \item[a)] Using Minkowski's inequality, $\Cpl(\mu,\nu)\subset \Pc_p(X\times X)$ for all $\mu,\nu\in \Pc_p(X)$.
	  \item[b)] Let $u\colon X\to \R$ and $L\geq 0$ with $|u(y)-u(z)|\leq L\|y-z\|$ for all $y,z\in X$.
	  Then, for all $\mu,\nu\in \Pc_p(X)$ and $\pi\in \Cpl(\mu,\nu)$, Minkowski's inequality imples that
	  \begin{align*}
	   \Bigg|\bigg(\int_X|u(y)|^p\,\mu(\d y)\bigg)^{˝\tfrac1p}-\bigg(\int_X|u(z)|^p\,\nu(\d z)\bigg)^{\tfrac1p}\Bigg|&\leq \bigg(\int_{X\times X}|u(y)-u(z)|^p\,\pi(\d y,\d z)\bigg)^{\tfrac1p}\\
	   	   &\leq L \bigg(\int_{X\times X}\| y-z\|^p\,\pi(\d y,\d z)\bigg)^{\tfrac1p}.
	  \end{align*}
	  In particular, the map $\Pc_p(X)\to \R, \;\mu\mapsto \int_X |u(y)|^p\,\mu(\d y)$ is continuous.
	  \item[c)] For $\mu\in \Pc_p(X)$, we have \[\Wc_p(\mu,\delta_0)=\bigg(\int_X \|y\|^p\, \mu(\d y)\bigg)^{\tfrac1p},\] where $\delta_0$ denotes the Dirac measure with barycenter $0$.
	  \item[d)]  For $\mu,\nu\in \Pc_p$, by the triangular inequality and H\"older's inequality,
	   \[
	  \int_X \|z\|\, \nu(\d z)\leq  \Wc_p(\mu,\nu)+ \int_X \|y\|\, \mu(\d y).
	  \]
	 \end{enumerate}
	\end{remark}

	On $X$, we consider a time-homogeneous Markov process with transition probabilities $(p_t)_{t\geq0}$ of the form
	\[
	p_t(x,B):=\mu_t\big(\{y\in X\colon\psi_t(x)+y\in B\}\big)\quad\mbox{for all }t\geq 0,\; x\in X,\text{ and }B\in\Bc(X),
	\]
	where $(\mu_{t})_{t \geq 0} \subset \Pc_{p}(X)$ is a family of probability measures and $ (\psi_{t})_{t \geq 0} $ is a family 
	of continuous maps $\psi_t\colon X \to X$. In particular, we assume that the transition probabilities satisfy the Chapman-Kolmogorov equations, i.e.,
	\begin{equation}\label{eq.consistency}
		\int_{X} u(\psi_{t+s}(x)+y_{t+s}) \,\mu_{t+s}(\d y_{t+s})= \int_{X} \int_{X} u\Big(\psi_{s}\big(\psi_{t}(x)+y_{t}\big)+y_{s}\Big)\,\mu_{s}(\d y_{s})\,\mu_{t}(\d y_{t})
	\end{equation}
	for all $s,t \geq 0$, $u\in \Cb$, and $x \in X$. In other words, the Markov process belonging to the family of transition kernels $(p_t)_{t\geq 0}$ is of the form
	\begin{equation}\label{eq.markovdecomp}
	\Xi_t^x=\psi_t(x)+Y_t\quad \text{with}\quad Y_t\sim \mu_t\quad \text{for all }t\geq 0\text{ and } x\in X.
	\end{equation}

	\begin{assumption}
		Throughout, we work under the following assumptions:
		\begin{enumerate}
			\item[(A1)] There exists a constant $c\geq 0$ such that, for all $t\geq 0$ and $x_1,x_2\in X$,
			\begin{equation}\label{eq.lipcond2}
				\|\psi_t(x_1)-\psi_t(x_2)\|\leq e^{c t}\|x_1-x_2\|. 
			\end{equation}
			\item[(A2)] We assume that $$K\psi_t(x)=\psi_t(Kx)\quad \text{for all }t\geq 0\text{ and }x\in X.$$ Moreover, $\psi_0(x)=x$ for all $x\in X$, and
			\[
			\lim_{h\downarrow 0}\sup_{\|x\|\leq r}\|\psi_h(Kx)-Kx\|=0\quad \text{for all }r\geq0.
			\]
			\item[(A3)] We assume that $\Wc_p(\mu_h,\delta_0)\to 0$ as $h\downarrow 0$ or, in other words,
			\[
			\lim_{h\downarrow 0}\int_X \|y\|^p\,\mu_h(\d y)=0.
			\]
		\end{enumerate}
	\end{assumption}
	
	We briefly discuss our assumptions in the following remark.
	
	\begin{remark}\label{rem.observations0}\
		\begin{enumerate}
			\item[a)] Note that our global Assumption (A2) implies that, for all $u\in \BUC$ and all $r\geq 0$,
			\[
			\sup_{\|x\|\leq r}\big|u\big(\psi_h(x)\big)-u(x)\big|\to 0\quad \text{as }h\downarrow 0.
			\]
			Indeed, for $u\in \Lipb$,
			\[
			 \sup_{\|x\|\leq r}\big|u\big(\psi_h(x)\big)-u(x)\big|\leq \|u\|_\Lip \sup_{\|x\|\leq r}\|\psi_h(Kx)-Kx\|\to 0 \quad\text{as }h\downarrow 0.
			\]			
			For general $u\in \BUC$, the statement can be obtained by the following argument that will be used in a similar form on various occasions. Let $u\in \BUC$. Then, there exists a sequence $(u_n)_{n\in\N}\in \Lipb$ with $u_n\to u$. By our global Assumptions (A1) and (A2), there exists some $h_0>0$ such that
			\begin{equation}\label{eq.Mr}
			M_r:=\sup_{h\in [0,h_0]}\sup_{\|x\|\leq r}\|\psi_h(x)\|\leq \sup_{h\in [0,h_0]}\|\psi_h(0)\|+e^{c h_0}r<\infty\quad \text{for all }r\geq 0.
			\end{equation}
			Therefore, for arbitrary $\ep>0$,
			\begin{align*}
			 \sup_{\|x\|\leq r}\big|u\big(\psi_h(x)\big)-u(x)\big|&\leq \sup_{\|x\|\leq r}\big|u_n\big(\psi_h(x)\big)-u_n(x)\big| +2\sup_{\|\xi\|\leq M_r}|u(\xi)-u_n(\xi)|\\
			 &< \sup_{\|x\|\leq r}\big|u_n\big(\psi_h(x)\big)-u_n(x)\big|+\ep
			\end{align*}
                        for $n\in \N$ sufficiently large. Letting, first, $h\downarrow 0$ and then $\ep\downarrow 0$, it follows that
                        \[
                        \lim_{h\downarrow 0}\sup_{\|x\|\leq r}\big|u\big(\psi_h(x)\big)-u(x)\big|= 0\quad \text{for all }r\geq0.
                        \]
			\item[b)] We assume that the family $(p_t)_{t\geq 0}$ of transition kernels is such that the related Markov process is of the form
			\[
			\Xi_t^x=\psi_t(x)+Y_t.
			\]
			However, also more general dynamics can be considered.\ For example, the additive operation $+$ could be the multiplication on the positive half line $(0,\infty)$, leading to geometric dynamics
			\[
			\Xi_t^x=xY_t.
			\]
			For more details, we refer to Section \ref{sec.discussion}, in particular, Example \ref{ex.GBM}.
			%\item[c)] Condition \eqref{eq.lipcond2} can be weakened to
			%\[
			%\|\psi_t(x_1)-\psi_t(x_2)\|\leq e^{L t}\|x_1-x_2\|\quad \text{for all }t\geq 0\text{ and }x_1,x_2\in X
			%\]
			%with an arbitrary constant $L\geq 0$, see Section \ref{sec.extA1}. 
			\item[c)] Consider the family $T=\big(T(t)\big)_{t\geq 0}$, given by
			\begin{equation}\label{Smu}
				\big(T(t)u\big)(x):=\int_X u(\xi) \, p_t(x,\d \xi)=\int_X u\big(\psi_t(x)+y\big)\,\mu_t(\d y)
			\end{equation}
			for all $t\geq 0$, $u\in \BUC$, and $x\in X$. Due to the Chapman-Kolmogorov equations \eqref{eq.consistency}, it follows that $T$ is a semigroup on $\BUC$. More precisely, it is the transition semigroup of the Markov process $(\Xi_t^x)$.\ Due to our global assumptions, the semigroup $T$ is strongly continuous on $\BUC$, see Remark \ref{remark.Smu}, below.
		\end{enumerate}
		
	\end{remark}
	
	Throughout, we consider a convex lower semicontinuous function $\phi\colon [0,\infty)\to [0,\infty]$ with $\phi(0)=0$ and $\phi(v)\neq 0$ for some $v> 0$. Note that $\phi$ is nondecreasing and continuous on ${\dom(\phi)}:=\{v\in [0,\infty)\colon \phi(v)<\infty\}$.\ Since $\phi$ is nondecreasing, either $\dom(\phi)=[0,\infty)$ or $\dom(\phi)=[0,a]$ for some $a\geq 0$. We further assume that the map $[0,\infty)\to [0,\infty],\; v\mapsto \phi\big(v^{1/p}\big)$ is convex.\ Since $\phi\not\equiv 0$, this implies that
	\begin{equation}\label{eq.asymptotic}
		\liminf_{v\to \infty}\frac{\phi(v)}{v^p}>0.
	\end{equation}
        In particular, the conjugate
	$$\phi^\ast(w):=\sup_{v\in[0,\infty)}\big(vw-\phi(v)\big),\quad \text{for }w\geq 0,$$ defines a continuous function $[0,\infty)\to [0,\infty)$ with $\phi^*(0)=0$. Typical examples for $\phi$ are $[0,\infty)\to [0,\infty), \;v\mapsto v^p$ or $\phi_a:=\infty\cdot \eins_{(a,\infty)}$, i.e., $\varphi_a(v)=0$, for $v\in[0,a]$,  and $\varphi_a(v)=\infty$, for $v>a$, with $a\geq 0$.
	
	\section{Wasserstein robust semigroups}\label{sec.wasserstein}
	
	In this section, we aim to study a distributionally robust version of the semigroup $T$ given in Equation \eqref{Smu}.\ We start with a preliminary version by considering the following family of operators:
	
	\begin{definition}
		For $t\geq 0$, $u\in \BUC$, and $x\in X$, we define
		\begin{equation}\label{eq.isometry}
			\big(I(t)u\big)(x)=\sup_{\nu\in \Pc_p(X)}\bigg(\int_X u\big(\psi_t(x)+z\big)\,\nu(\d z)-\phi_t\big(\Wc_p(\mu_t,\nu)\big)\bigg),
		\end{equation}
		where  $\phi_t(v):=t\phi\big(e^{-c t}\frac{v}{t}\big)$, for $v\geq 0$ and $t>0$, and $\varphi_0:=\infty\cdot \eins_{(0,\infty)}$.
	\end{definition}
	
	\begin{remark}\label{rem.basic}\
	\begin{enumerate}
	 \item[a)] Let $t\geq 0$ and $u_1,u_2\in \BUC$. Then,
         \begin{equation}\label{eq.rembasic1}
	  \|I(t)u_1-I(t)u_2\|_\infty\leq \|u_1-u_2\|_\infty.
	 \end{equation}
	 Moreover, $I(t)0=0$, which implies that $\|I(t)u\|_\infty\leq \|u\|_\infty$ for all $u\in \BUC$. In fact,
		\begin{align*}
			\big|\big(I(t)u_1\big)(x)-\big(I(t)u_2\big)(x)\big|&\leq \sup_{\nu\in \Pc_p(X)}\int_X \big|u_1\big(\psi_t(x)+z\big)-u_2\big(\psi_t(x)+z\big)\big|\, \nu(\d z)\\
			&\leq  \|u_1-u_2\|_\infty\quad \text{for all }x\in X.
		\end{align*}
		By taking the supremum over all $x\in X$, the inequality \eqref{eq.rembasic1} follows.
	 \item[b)] Let $t\geq 0$ and $u\in \Lipb$. Then,
	 \[
	 \|I(t)u\|_{\Lip}\leq e^{c t}\|u\|_\Lip.
	 \]
	 In particular, $I(t)u\in \Lipb$ for all $u\in \Lipb$. Indeed, due to our global assumptions (A1) and (A2), 
		\begin{align*}
			\big|\big(I(t)u\big)(x_1)-\big(I(t)u\big)(x_2)\big|&\leq \sup_{\nu\in \Pc_p(X)}\int_X \big| u\big(\psi_t(x_1)+z\big)-u\big(\psi_t(x_2)+z\big)\big|\, \nu(\d z)\\
			&\leq e^{c t}\|u\|_\Lip \|K(x_1-x_2)\|\quad \text{for all }x_1,x_2\in X.
		\end{align*}
	\end{enumerate}	 
	\end{remark}

	We start with the following observations.
	
	\begin{lemma}\label{lemma1}
	 For all $a\geq 0$, $t\geq 0$, $u\in \Lip$, and $x\in X$, let
	 \begin{equation}\label{eq.lem10}
	  \big(I_a(t)u\big)(x):=\sup_{\Wc_p(\mu_t,\nu)\leq at} \int_X u\big(\psi_t(x)+z\big)\, \nu(\d z).
	 \end{equation}
	 For every $L\geq 0$, there exists a constant $a\geq 0$, depending only on $\phi$, $c$, and $L$ such that
         \begin{enumerate}
          \item[a)] for every $u\in \Lip$ with $\|u\|_{\Lip}\leq L$,
          \begin{equation}\label{eq.lem1}
           \big(I(t)u\big)(x)=\sup_{\Wc_p(\mu_t,\nu)\leq at} \bigg(\int_X u\big(\psi_t(x)+z\big)\, \nu(\d z)-\phi_t\big(\Wc_p(\mu_t,\nu)\big)\bigg)
          \end{equation}
          for all $t\in [0,1]$ and $x\in X$,
          \item[b)] for all $u_1,u_2\in \Lip$ with $\|u_i\|_\Lip\leq L$ for $i=1,2$,
          \begin{equation}\label{eq.lem2}
	  I(t)u_1 -I(t)u_2\leq I_a(t)(u_1-u_2) \quad \text{for all }t\in[0,1].
	 \end{equation}
         \end{enumerate}
	\end{lemma}
	
	\begin{proof}
          Let $L\geq 0$ and $$J:=\big\{v\in [0,\infty)\colon \phi(e^{-c}v)\leq 1+Lv\big\}.$$ By Equation  \eqref{eq.asymptotic}, $J=[0,a]$ for some $a\geq 0$, depending only on $\phi$, $c$, and $L$. Now, let $u\in \Lip$ with $\|u\|_{\Lip}\leq L$, $t\geq 0$, and $x\in X$. For $t=0$, the equality stated in \eqref{eq.lem1} is trivial.\ Therefore, assume that $t\in(0,1]$. Then, there exists some $\nu\in \Pc_p(X)$ with
          \[
           \int_X u\big(\psi_t(x)+y\big)\, \mu_t(\d y)\leq \big(I(t)u\big)(x)\leq t+\int_X u\big(\psi_t(x)+z\big)\, \nu(\d z)-\phi_t\big(\Wc_p(\mu_t,\nu)\big),
          \]
         which leads to the inequality
         \begin{align*}
           \phi\Big(e^{-c t}\tfrac{\Wc_p(\mu_t,\nu)}{t}\Big)&\leq 1+\frac{1}{t}\int_X\int_X \Big(u\big(\psi_t(x)+z\big)-u\big(\psi_t(x)+y\big)\Big)\,\nu(\d z)\, \mu_t(\d y)\\
          &\leq 1+L\tfrac{\Wc_p(\mu_t,\nu)}{t}.
         \end{align*}
         Therefore, since $t\in(0,1]$, it is sufficient to take the supremum in \eqref{eq.isometry} over the set
         \begin{align*}
         \Big\{\nu\in \Pc_p(X)\colon \tfrac{\Wc_p(\mu_t,\nu)}{t}\in J\Big\}&=\Big\{\nu\in \Pc_p(X)\colon \tfrac{\Wc_p(\mu_t,\nu)}{t}\in [0,a] \Big\}.
         \end{align*}
          The proof of part a) is complete. In order to prove part b), let $u_1,u_2\in \Lip$ with $\|u_i\|_\Lip\leq L$ for $i=1,2$, $t\in[0,1]$, and $x\in X$. Then, by part a),
         \[
          \big(I(t)u_1\big)(x) -\big(I(t)u_2\big)(x)\leq \sup_{\Wc_p(\mu_t,\nu)\leq at}\int_X\Big(u_1\big(\psi_t(x)+z\big)-u_2\big(\psi_t(x)+z\big)\Big)\,\nu(\d z).
         \]
         The proof is complete.
	\end{proof}
	
	Note that, for $u\in \Lipb$, the definition of $I_a(t)u$ by virtue of \eqref{eq.lem10} is exactly the same as \eqref{eq.isometry} with $\phi=\phi_a=\infty\cdot \eins_{(a,\infty)}$.
	
	\begin{lemma}\label{lem.Ib}
	 For every $C\geq 0$, there exists a constant $b\geq 0$, depending only on $\varphi$, $c$, $p$, and $C$ such that
	 \[
	  I(t)u=\sup_{\Wc_p(\mu_t,\nu)\leq bt^\alpha}\bigg(\int_X u\big(\psi_t(x)+z\big)\, \nu(\d z)-\phi_t\big(\Wc_p(\mu_t,\nu)\big)\bigg)
	 \]
        for all $u\in \BUC$ with $\|u\|_\infty\leq C$, $t\in [0,1]$, and $x\in X$, where $\alpha:=\frac{p-1}{p}$.
	\end{lemma}
	
	\begin{proof}
	 By \eqref{eq.asymptotic}, there exists some constant $M>0$, depending only on $\phi$ and $p$, such that $$v^p\leq M\big(1+\phi(v)\big)\quad \text{for all }v\geq 0.$$
	 Let $C\geq 0$, $u\in \BUC$ with $\|u\|_\infty\leq C$, $t\in [0,1]$, and $x\in X$. For $t=0$, the statement is trivial.\ Therefore, assume that $t>0$. Then, there exists some $\nu\in \Pc_p(X)$ with
          \[
           \int_X u\big(\psi_t(x)+y\big)\, \mu_t(\d y)\leq \big(I(t)u\big)(x)\leq 1+\int_X u\big(\psi_t(x)+z\big)\, \nu(\d z)-\phi_t\big(\Wc_p(\mu_t,\nu)\big),
          \]
         which leads to the inequality
         \[
          \phi_t\big(\Wc_p(\mu_t,\nu)\big)\leq 1+2\|u\|_\infty\leq 1+2C.
         \]
         Therefore, since $t\in [0,1]$, it is sufficient to take the supremum in \eqref{eq.isometry} over the set
         \[
          \Big\{\nu\in \Pc_p(X)\colon \phi\Big(e^{-c}\tfrac{\Wc_p(\mu_t,\nu)}{t}\Big)\leq \tfrac{1+2C}{t}\Big\}\subset \Big\{\nu\in \Pc_p(X)\colon \Wc_p(\mu_t,\nu)\leq bt^\alpha\Big\}
         \]
         with $b:=e^{c}\big(2M(1+C)\big)^{1/p}$.
	\end{proof}

	\begin{lemma}\label{lem.technical}
	 For $i=1,2$, let $(u_n^i)_{n\in \N}\subset \BUC$ with $(u_n^1-u_n^2)\to 0$ as $n\to \infty$. Then, there exists some $h_0>0$ such that $\big(I(h)u_n^1-I(h)u_n^2\big)\to 0$ uniformly in $h\in [0,h_0]$ as $n\to \infty$, i.e., $\sup_{n\in \N}\sup_{h\in [0,h_0]}\|I(h)u_n^1-I(h)u_n^2\|_\infty<\infty$ and 
	 \begin{equation}\label{eq.convI}
	 \lim_{n\to \infty} \sup_{h\in [0,h_0]}\sup_{\|x\|\leq r} \big|\big(I(h)u_n^1\big)(x)-\big(I(h)u_n^2\big)(x)\big|=0\quad \text{for all }r\geq 0.
	 \end{equation}
        \end{lemma}
        
         \begin{proof}
          By Equation \eqref{eq.rembasic1}, we have $$\big\|I(t)u_n^1-I(t)u_n^2 \big\|_\infty\leq \|u_n^1-u_n^2\|_\infty\quad \text{for all }t\geq 0\text{ and }n\in \N.$$ Therefore, it remains to verify \eqref{eq.convI}. Let $u_n:=u_n^1-u_n^2$ and $\alpha:=\frac{p-1}{p}$. By Lemma~\ref{lem.Ib} with $C:= \sup_{n\in \N}\|u_n\|_\infty$, there exists some $b\geq 0$ such that
          \[
           \big|\big(I(t)u_n^1\big)(x)-\big(I(t)u_n^2\big)(x)\big|\leq \sup_{\Wc_p(\mu_t,\nu)\leq bt^\alpha} \int_X \big|u_n\big(\psi_t(x)+z\big)\big|\,\nu(\d z)
          \]
          for all $t\in [0,1]$ and $x\in X$. By our global Assumptions (A2) and (A3), there exists some $h_0\in (0,1]$ such that
          \[
           \sup_{h\in [0,h_0]}\|\psi_h(0)\|<\infty\quad \text{and}\quad \sup_{h\in [0,h_0]} \int_X \|y\|^p\, \mu_h(\d y)\leq 1.
          \]
          %Then, by Assumption (A1),
          %\[
          % M_r:=\sup_{h\in [0,h_0]}\sup_{\|x\|\leq r} \|\psi_h(x)\|\leq \sup_{h\in [0,h_0]}\|\psi_h(0)\|+r<\infty\quad \text{for all }r\geq 0.
          %\]
          Let $\ep>0$, $r\geq 0$, $M_r$ be given by Equation \eqref{eq.Mr}, and $M>0$ with $\frac{C(1+bh_0^\alpha)}{M}<\ep$.\ Then, using Markov's inequality and Remark \ref{rem.Wp} d),
          \begin{align*}
           \sup_{h\in [0,h_0]}\sup_{\|x\|\leq r} \sup_{\Wc_p(\mu_h,\nu)\leq bh^\alpha}& \int_X \big|u_n\big(\psi_h(x)+z\big)\big|\,\nu(\d z)\\
            &\leq \sup_{\|\xi\|\leq M+M_r}|u_n(\xi)|+\sup_{h\in [0,h_0]}\sup_{\Wc_p(\mu_h,\nu)\leq bh^\alpha}\frac{C}{M}\int_X\|z\|\,\nu(\d z)\\
           &\leq \sup_{\|\xi\|\leq M+M_r}|u_n(\xi)|+\frac{C\big(1+bh_0^\alpha\big)}{M}<\ep
          \end{align*}
          for $n\in \N$ sufficiently large.
         \end{proof}
         
              \begin{remark}\label{rem.centerpoint}
         	For $u\in \Lip$,
         	         	\begin{equation}\label{eq.difference}
         		\big\|I(h)u-T(h)u\big\|_\infty\leq h\phi^*\big(e^{c h}\|u\|_\Lip\big).
         	\end{equation}
         In fact, for all $x\in X$, by Remark \ref{rem.Wp} b),
         	\[
         	\big(I(h)u\big)(x)-\big(T(h)u\big)(x)\leq \sup_{\nu\in \Pc_p}\|u\|_\Lip \Wc_p(\mu_h,\nu)-\phi_h\big(\Wc_p(\mu_h,\nu)\big)=h\phi^*\big(e^{c h}\|u\|_\Lip\big).
         	\]
         \end{remark}
         
         \begin{lemma}\label{lem.keyestimates}
          Let $h_0>0$ as in Lemma~\ref{lem.technical}.\ Then, the operator $I(h)\colon \BUC \to \BUC$ is well-defined, convex, and monotone for all $h\in [0,h_0]$. Moreover, for all $u\in \BUC$,
          \[
           I(h)u\to u\quad \text{as }h\downarrow 0.
          \]
         \end{lemma}

     \begin{proof}
       By Remark \ref{rem.basic}, $I(t)u\in \Lipb$ for all $t\geq 0$ and $u\in \Lipb$. By Lemma~\ref{lem.technical}, it follows that $I(h)\colon \BUC\to \BUC$ is well-defined for all $h\in [0,h_0]$, since, by definition of $\BUC$, every $u\in \BUC$ can be approximated by a sequence $(u_n)_{n\in \N}\subset \Lipb$ in the sense that $u_n\to u$ as $n\to \infty$.\ The convexity and monotonicity of the operator $I(h)$, for $h\in [0,h_0]$, are an immediate consequence of its definition. Let $u\in \Lipb$ and $r\geq 0$. %Then, by Remark \ref{rem.centerpoint}, for all $x\in X$,...
		%\[
		%\big(I(h)u\big)(x)-\big(T(h)u\big)(x)\leq \sup_{\nu\in \Pc_p}\|u\|_\Lip \Wc_p(\mu_h,\nu)-\phi_h\big(\Wc_p(\mu_h,\nu)\big)=he^{c h}\phi^*\big(\|u\|_\Lip\big),
		%\]
		%which, by taking the supremum over all $x\in X$, leads to the inequality
		%\begin{equation}\label{eq.difference}
		%	\big\|I(h)u-T(h)u\big\|_\infty\leq he^{c h}\phi^*\big(\|u\|_\Lip\big).
		%\end{equation}
		By Remark \ref{rem.observations0} a),
		\[
		\sup_{\|x\|\leq r}\big|u\big(\psi_h(x)\big)-u(x)\big|\to 0\quad \text{as }h\downarrow 0.
		\]
		By virtue of Assumption (A3), we may conclude that
		\begin{align*}
		\sup_{\|x\|\leq r}\big|\big(T(h)u\big)(x)-u(x)\big|&\leq \|u\|_\Lip\bigg(\int_X \|y\|^p\, \mu_h(\d y)\bigg)^{1/p}\\
		&\quad+\sup_{\|x\|\leq r}\big|u\big(\psi_h(x)\big)-u(x)\big|\to 0\quad \text{as }h\downarrow 0,
		\end{align*}
		which, together with \eqref{eq.difference}, implies that $I(h)u\to u$ as $h\downarrow 0$. For general $u\in\BUC$, the strong continuity 
		follows by approximating with $(u_n)_{n\in \N}\subset \Lipb$ and using the convergence $I(h)u_n\to I(h)u$, uniformly in $h\in [0,h_0]$, as $n\to \infty$, see Lemma~\ref{lem.technical}.
     \end{proof}

    Alternatively, the family of operators $I=\big(I(t)\big)_{t\geq 0}$ can also be defined via couplings with first marginals given in terms of the family of laws $(\mu_t)_{t\geq 0}$, as we point out in the following remark.
	
	\begin{remark}\label{rem.Icoupling}
		Let
		\begin{equation}\label{eq.deltap}
			\Delta_p\pi:=\bigg(\int_{X\times X} \|y-z\|^p\, \pi(\d y,\d z)\bigg)^{1/p}\quad \text{for all }\pi\in \Pc_p(X\times X).
		\end{equation}
		Then, by definition, $\Wc_p(\mu,\nu)=\inf_{\pi\in \Cpl(\mu,\nu)} \Delta_p\pi$ for all $\mu,\nu\in \Pc_p(X)$.\ Moreover, by Remark \ref{rem.Wp} b), $$| \Delta_p\pi-  \Delta_p\pi^\prime |\le \Wc_p(\pi,\pi^\prime)\quad\text{for all }\pi'\in \Pc_p(X\times X),$$ i.e., the map $\Pc_p(X\times X)\to \R,\; \pi\mapsto \Delta_p\pi$ is $1$-Lipschitz.\ For $\pi\in \Pc_p(X\times X)$ and $i=1,2$, we denote by $\pi_i$ the $i$-th marginal of $\pi$. Using these notations, we find that, for all $t\geq0$, $u\in \BUC$, and $x\in X$,
		\begin{equation}\label{eq.Icoupling}
			\big(I(t)u\big)(x)=\sup_{\substack{\pi\in \Pc_p(X\times X)\\ \pi_1=\mu_t}}\bigg(\int_X u\big(\psi_t(x)+z\big)\,\pi_2(\d z)-\phi_t\big(\Delta_p\pi\big)\bigg).
		\end{equation}
		In fact, since $\phi$ is nondecreasing, $\geq $ holds in \eqref{eq.Icoupling}.\ In order to establish the other inequality, let $\nu\in \Pc_p(X)$.\ By \cite[Theorem 4.1]{villani2008optimal}, there exists an optimal coupling $\pi^*\in \Cpl(\mu,\nu)$ with $\Wc_p(\mu_t,\nu)=\Delta_p\pi^*$.\ By definition of $\Cpl(\mu_t,\nu)$, it follows that $\pi^*_1=\mu_t$ and $\pi^*_2=\nu$. Therefore,
		\begin{align*}
			\int_X u\big(\psi_t(x)+z\big)\,\nu(\d z)-\phi_t\big(\Wc_p&(\mu_t,\nu)\big) =\int_X u\big(\psi_t(x)+z\big)\,\pi^*_2(\d z)-\phi_t\big(\Delta_p\pi^*\big)\\
			&\leq \sup_{\substack{\pi\in \Pc_p(X\times X)\\ \pi_1=\mu_t}}\bigg(\int_X u\big(\psi_t(x)+z\big)\,\pi_2(\d z)-\phi_t\big(\Delta_p\pi\big)\bigg).
		\end{align*}
		Taking the supremum over all $\nu\in \Pc_p(X)$ yields the desired claim.
	\end{remark}
	
	The following measurable selection result forms the basis for several proofs in this section.
	
	\begin{lemma}\label{lem.measselection}
		Let $t\geq 0$, $u\in \BUC$, and $\ep>0$. Then, there exists a measurable map $\pi\colon X\to \Pc_p(X\times X),\; x\mapsto \pi^x$ such that, for each $x\in X$, the first marginal $\pi^x_1$ of $\pi^x$ equals $\mu_t$ and
		\[
		\big(I(t)u\big)(x)\leq \int_X u\big(\psi_t(x)+z\big)\,\pi_2^x(\d z)-\phi_t\big(\Delta_p\pi^x\big)+\ep,
		\]
		where  $\pi^x_2\in \Pc_p(X)$ denotes the second marginal of $\pi^x$.
	\end{lemma}
	
	\begin{proof}
		Let $D_t:=\{\pi\in \Pc_p(X\times X)\colon \pi_1=\mu_t\text{ and }\phi_t(\Delta_p\pi)<\infty\}$. Then, by Remark \ref{rem.Icoupling}, the map
		\[
		f\colon X\times D_t\to \R,\quad (x,\pi)\mapsto  \int_{X\times X} u\big(\psi_t(x)+z\big)\,\pi_1 (\d z)-\phi_t\big(\Delta_p\pi\big)
		\]
		is continuous, and 
		\[
		\big(I(t)u\big)(x)=\sup_{\pi\in D_t}f(x,\pi)\quad \text{for all }x\in X.
		\]
		The statement now follows from \cite[Proposition 7.34]{bertsekas2004stochastic}.
	\end{proof}
	
	\begin{lemma}\label{lem.monotonicity}
		Let $ s,t\geq 0$ and $u\in \BUC$. Then,
		\[ \big(I(t) I(s) u\big)(x) \leq \big(I(t+s) u\big)(x)\quad \text{for all }x\in X. \]
	\end{lemma}
	
	\begin{proof}
		Let $\ep>0$, $ x \in X $, and $ \pi_t\in \Pc_p(X\times X) $ with first marginal $\pi_{t,1}=\mu_t$ and
		\[
		\big(I(t) I(s) u\big)(x) \leq  \int_{X} \big(I(s) u\big)\big(\psi_{t}(x)+z_t\big)\nu_t(\d z_t) - \phi_{t}\big(\Delta_p \pi_t\big)+\frac{\ep}{2},
		\]
		where $\nu_t:=\pi_{t,2}\in \Pc_p(X)$ is the second marginal of $\pi_t$.
		By Lemma~\ref{lem.measselection}, there exists a measurable map $\pi_s\colon X\to \Pc_p(X\times X),\; z\mapsto \pi_s^z$ such that, for all $z\in X$, the first marginal $\pi_{s,1}^z$ of $\pi_s^ z$ is $\mu_s$ and 
		\begin{align*}
			\big(I(s) u\big)\big(\psi_{t}(x)+z\big) \leq \int_{X} u\Big(\psi_{s}\big(\psi_{t}(x)+ z\big)+z_s\Big) \nu^{z}_{s}(\d z_s)- \phi_{s}\big(\Delta_p \pi_s^z\big)+\frac{\ep}{2}
		\end{align*}
		with $\nu_s^z:=\pi_{s,2}^z\in \Pc_p(X)$ being the second marginal of $\pi_s^z$.
		It follows that
		\begin{align}
			\notag\big(I(t) I(s) u\big)(x)\leq \int_{X} & \int_{X} u\Big(\psi_{s}\big(\psi_{t}(x)+z_t\big)+z_s\Big)\, \nu_{s}^{z_t}(\d z_s)\, \nu_t(\d z_t)\\
			\label{eq.consist1}&- \int_{X} \phi_{s}\big(\Delta_p\pi^{z_t}_s\big)\, \nu_t(\d z_t)- \phi_{t}\big(\Delta_p \pi_t\big)+\ep,
		\end{align}
		where the measurability of the map $z\mapsto \phi_{s}\big( \Delta_p\pi_{s}^{z}\big)$ follows from the measurability of the map $z\mapsto \pi^ z$ and the lower semicontinuity of the map $\pi\mapsto \phi_s\big(\Delta_p\pi\big)$. We now define a new coupling $ \pi_{t+s}$ via
		\begin{align}
			\notag \int_{X\times X} g(y_{t+s},z_{t+s})\, \pi_{t+s}(\d &y_{t+s},\d z_{t+s})=\int_{X\times X} \int_{X\times X} g\Big(\psi_{s}\big(\psi_{t}(x)+y_t\big)+y_s-\psi_{t+s}(x),\\
			\label{eq.coupling} &\psi_{s}\big(\psi_{t}(x)+z_t\big)+z_s-\psi_{t+s}(x)\Big) \,\pi_{s}^{z_t}(\d y_s, \d z_s)\,\pi_{t}(\d y_t, \d  z_t)
		\end{align}
		for all bounded continuous functions $g\colon X\times X\to \R$. Note that the outer integral on the right-hand side of \eqref{eq.coupling} is well-defined since the inner integral depending on $y_t$ and $z_t$ is jointly measurable in $(y_t,z_t)$. The latter follows from the measurability of the map $z\mapsto \pi^ z$, the choice of $g$ as a bounded continuous function, and the (Lipschitz) continuity of $\psi_s$.  
		Then, the consistency condition (\ref{eq.consistency}) and the definition of $\pi_{t+s}$ in \eqref{eq.coupling} with $g(y,z)= u(\psi_{t+s}(x)+y)$, for $y,z\in X$, imply that the first marginal of $\pi_{t+s}$ equals to $\mu_{t+s}$, i.e., $\pi_{1,t+s}=\mu_{t+s}$. In particular, for $g(y,z)= u(\psi_{t+s}(x)+z)$ with $y,z\in X$ in \eqref{eq.coupling}, we have
		\begin{align*}
			\big(I(t+s)u\big)(x) &\geq  \int_{X}  u\big(\psi_{t+s}(x)+z_{t+s}\big)\,\nu_{t+s}(\d z_{t+s}) - \phi_{t+s}\big(\Delta_p\pi_{t+s}\big)\\
			&=\int_{X} \int_{X} u\Big(\psi_{s}\big(\psi_{t}(x)+z_t\big)+z_s\Big) \,\nu_{s}^{z_t}(\d z_s)\,\nu_t(\d z_t)-\phi_{t+s}\big(\Delta_p\pi_{t+s}\big),
		\end{align*}
		where $\nu_{t+s}:=\pi_{2,t+s}$ denotes the second marginal of $\pi_{t+s}$. Note that $\pi_{t+s}$ is an element of $\Pc_p(X\times X)$. This follows from $\pi_{1,t+s}=\mu_{t+s}$ together with the estimate 
		\begin{align}
			\label{eq.keyest} \phi_{t+s}\big(\Delta_p\pi_{t+s}\big)\leq  \phi_{t}&\big(\Delta_p\pi_t\big)+\int_{X} \phi_{s}\big(\Delta_p\pi_s^{z_t}\big)\,\nu_t(\d z_t).
		\end{align}
		It remains to prove \eqref{eq.keyest}, since, in view of \eqref{eq.consist1}, the estimate \eqref{eq.keyest} implies $I(t+s)u\geq I(t)I(s)u$ after taking the limit $\ep\downarrow 0$. We start by showing that  
		\begin{align*} \Delta_p\pi_{t+s}
			\leq \Delta_p\pi_{t} + \bigg(\int_{X}\int_{X\times X} \|y_s-z_s\|^p\, \pi_s^{z_t}(\d y_s,\d z_s)\,\nu_t(\d z_t)\bigg)^{1/p}. \end{align*}
		Using \eqref{eq.coupling}, Minkowski's inequality, and Assumption (A1),
		\begin{align*}
			\Delta_{p}\pi_{t+s}&= \bigg(\int_{X\times X}\|y_{t+s}-z_{t+s}\|^p\, \pi_{t+s}(\d y_{t+s},\d z_{t+s})\bigg)^{1/p}\\
			&=\left(\int_{X\times X} \big\|\psi_{t+s}(x)+y_{t+s}-\big(\psi_{t+s}(x)+z_{t+s}\big)\big\|^{p}\,\pi_{t+s}(\d y_{t+s},\d z_{t+s})\right)^{1/p}\\
			& \leq   \bigg(\int_{X\times X}\int_{X\times X} \big\|\psi_{s}\big(\psi_{t}(x)+y_t\big)-\psi_{s}\big(\psi_{t}(x)+z_t\big)\big\|^{p}\,\pi_{t}(\d y_t,\d z_t)\bigg)^{1/p} \\
			&  \qquad+ \bigg(\int_{X}\int_{X\times X}\|y_s-z_s\|^{p}  \,\pi_{s}^{z_t}(\d y_s,\d z_s)\,\nu_{t}(\d z_t)\bigg)^{1/p}\\
			&\leq e^{c s}\Delta_p\pi_t+ \bigg(\int_{X}\int_{X\times X}\|y_s-z_s\|^{p}\,  \pi_{s}^{z_t}(\d y_s,\d z_s)\,\nu_{t}(\d z_t)\bigg)^{1/p}.
		\end{align*}
		Note that the function $z\mapsto \int_{X\times X}\|y_s-z_s\|^{p} \, \pi_{s}^{z}(\d y_s,\d z_s)$ is in fact measurable, since the map $z\mapsto \pi_s^z$ is measurable and the map $\pi\mapsto \int_{X\times X}\|y-z\|^p\, \pi(\d y,\d z)$ is continuous, by Remark \ref{rem.Wp} b), and thus measurable. Finally, we show that \eqref{eq.keyest} holds.
%		\begin{align*}
%			\phi_{t+s}\big( \Delta_p\pi_{t+s} \big)\leq  \phi_{t}\big(\Delta_p\pi_t\big) + \int_{X}\phi_{s}\big( \Delta_p\pi_{s}^{z_t}\big)\, \nu_{t}(\d z_t).
%		\end{align*}
		Since we assumed $ v \mapsto \phi\big(v^{1/p}\big) $ to be convex, the map $ v \mapsto \phi_{s}\big(v^{1/p}\big) $ is convex as well, which, by Jensen's inequality, implies that
		\[
		\phi_s\Bigg(\bigg(\int_{X} \int_{X\times X} \|y_s-z_s\|^p\, \pi_s^{z_t}(\d y_s,\d z_s)\,\nu_t(\d z_t)\bigg)^{1/p}\Bigg)\leq \int_{X}\phi_{s}\big(\Delta_{p}\pi_s^{z_t}\big)\,\nu_{t}(\d z_t).
		\]
		The convexity of $ \phi $ yields that
		\begin{align*}
			\phi_{t+s}\big(\Delta_p\pi_{t+s}\big)&=(t+s) \phi\bigg(e^{-c (t+s)}\frac{\Delta_p\pi_{t+s}}{t+s}\bigg)\leq  \phi_{t}\big(\Delta_p\pi_t\big) + \int_{X}\phi_{s}\big( \Delta_{p}\pi_s^{z_t} \big)\,\nu_{t}(\d z_t).
		\end{align*}
		The proof is complete.
	\end{proof}

	By $\BUC^1$, we denote the space of all $u\in \Lipb$ with Fr\'echet derivative $u'\colon X\to X'$ satisfying $x\mapsto \|u'(x)\|\in \BUC$. Here and throughout, the dual space $X'$ of $X$ is endowed with the operator norm 
	\[
	\|x'\|:=\sup_{\|x\|\leq 1} |x'x|\quad \text{for all }x'\in X'.
	\]
	
	\begin{lemma}\label{lem.keygenerator}
		Let $u\in \BUC^1$. Then,
		\begin{equation}\label{eq.difference1}
		\frac{I(h)u-T(h)u}{h}\to \phi^*\big(\|u'(\,\cdot\,)\|\big) \quad \text{as }h\downarrow 0.
		\end{equation}
		In particular, the map $X\to \R,\; x\mapsto \phi^*\big(\|u'(x)\|\big)$ is an element of $\BUC$.
	\end{lemma}
	
	\begin{proof}
	By Remark \ref{rem.centerpoint}, $$\sup_{h\in[0,1]}\bigg\|\frac{I(h)u-T(h)u}{h}\bigg\|_{\infty}<\infty,$$
	since $u\in \BUC^1\subset \Lipb$. Let $r\geq 0$. We start by showing that, for $h>0$ sufficiently small,
		\begin{equation}\label{eq.keygenerator1}
		\sup_{\|x\|\leq r}\bigg(\phi^ *\big(\|u'(x)\|\big)-  \frac{\big(I(h)u\big)(x)-\big(T(h)u\big)(x)}{h}\bigg)\leq \ep.
		\end{equation}
		For $h\geq 0$ and $\theta\in X$, let $\nu_{h,\theta}:=\mu_h\ast \delta_{h\theta}$, where $\ast$ denotes the convolution of two probability measures. Clearly, $\Wc_p(\mu_h,\nu_{h,\theta})\leq h\|\theta\|$, and therefore
		\begin{align}
		\notag	&\frac{\big(I(h)u\big)(x)-\big(T(h)u\big)(x)}{h}\geq \frac{1}h\bigg(\int_X u\big(\psi_h(x)+z\big)\, \nu_{h,\theta}(\d z)-\phi_h(h\|\theta\|)-\big(T(h)u\big)(x)\bigg)\\
		\label{eq.keygenerator2}	&\qquad \qquad \qquad\qquad \;\;\;=\int_X \frac{u\big(\psi_h(x)+y+h\theta\big)-u\big(\psi_h(x)+y\big)}{h}\, \mu_h(\d y) -\phi\big(e^{-c h}\|\theta\|\big)
		\end{align}
		for all $x,\theta\in X$ and $h\geq 0$. Since $u'$ is bounded, our assumptions on $\phi$ imply that there exists some constant $a\geq 0$ such that
		\[
		\phi^*\big(\|u'(x)\|\big)=\sup_{v\in [0,a]} \big\{v\|u'(x)\|-\phi(v)\big\}=\sup_{\|\theta\|\leq a} \big\{u'(x)\theta-\phi(\|\theta\|)\big\}\quad \text{for all }x\in X.
		\]
		Fix $\ep>0$. Using Taylor's theorem, the uniform continuity of $u'$ on bounded subsets of $X$, and Remark \ref{rem.observations0} a), there exists some $\delta>0$ such that
		\begin{equation}\label{eq.keygen1}
			u'(x)\theta\leq \frac{u\big(\psi_h(x)+y+h\theta\big)-u\big(\psi_h(x)+y\big)}{h}+\frac{\ep}{2}
		\end{equation}
		for all $x,y,\theta \in X$ with $\|x\|\leq r$, $\|y\|\leq \delta$, and $\|\theta\|\leq a$, and $h>0$ sufficiently small. Moreover, since $u'$ is bounded, it follows from Markov's inequality that 
		\begin{equation}\label{eq.keygen2}
			\bigg|\int_{\{\|y\|> \delta\}}\frac{u\big(\psi_h(x)+y+h\theta\big)-u\big(\psi_h(x)+y\big)}{h}\,\mu_h(\d y)\bigg|\leq \frac{a \|u'\|_\infty}{\delta} \Wc_p(\mu_h,\delta_0)< \frac\ep2
		\end{equation}
		for all $x,\theta\in X$ with $\|x\|\leq r$ and $\|\theta\|\leq a$, and $h>0$ sufficiently small. Therefore, since $\mu_h(\{\|y\| \le\delta\})\to 1$ as $h\downarrow 0$ by Assumption (A3), a combination of \eqref{eq.keygenerator2}, \eqref{eq.keygen1}, and \eqref{eq.keygen2} together with $e^{-c h}\leq 1$, for all $h\geq  0$, yields that
		\[
		u'(x)\theta-\phi(\|\theta \|)\leq \int_X \frac{u\big(\psi_h(x)+y+h\theta\big)-u\big(\psi_h(x)+y\big)}{h}\, \mu_h(\d y)-\phi\big(e^{-c h}\|\theta \|\big)+\ep
		\]
		for all $x,\theta\in X$ with $\|x\|\leq r$ and $\|\theta\|\leq a$, and $h>0$ sufficiently small. Taking the supremum over all $x,\theta \in X$ with $\|x\|\leq r$ and $\|\theta\|\leq a$, \eqref{eq.keygenerator1} follows.\\
		It remains to show that, for $h>0$ sufficiently small,
		\begin{equation}\label{eq.keygenerator3}
		 \sup_{\|x\|\leq r}\bigg(\frac{\big(I(h)u\big)(x)-\big(T(h)u\big)(x)}{h}- \phi^ *\big(\|u'(x)\|\big)\bigg)\leq \ep.
		\end{equation}
		To that end, for $x\in X$ and $h\geq 0$, we consider the auxiliary function $g_{h,x}\colon X\times X\to \R$, defined by
		\[
		g_{h,x}(y,z):=\frac{\big|u\big(\psi_h(x)+y\big)-u\big(\psi_h(x)+z\big)\big|}{\|y-z\|}\quad \text{for all }y,z\in X.
		\]
		Let $\ep>0$, $b\geq 0$ as in Lemma~\ref{lem.Ib} with $C:=\|u\|_\infty$, and $\alpha:=\frac{p-1}{p}$. Then, by Lemma~\ref{lem.Ib}, for all $h\in [0,1]$ and $x\in X$, there exists some $\nu_h^x \in \Pc_p(X)$ with $\Wc_p(\mu_h,\nu_h^x)\leq bh^\alpha$ and 
		\begin{equation}\label{eq.keygenerator4}
		\big(I(h)u\big)(x)\leq \frac{\ep h}{2}+\int_X u\big(\psi_h(x)+z\big)\, \nu_h^x(\d z)-\phi_h\big(\Wc_p(\mu_h,\nu_h^x)\big).
		\end{equation}
		For each $h\in [0,1]$ and $x\in X$, let $\pi_h^x\in \Pc_p(X\times X)$ be an optimal coupling between $\mu_h$ and $\nu_h^x$. Then,
		\begin{align*}
			&\frac{\big(I(h)u\big)(x)-\big(T(h)u\big)(x)}{h}\\
			& \qquad\quad\leq \frac{1}{h}\int_{X\times X}g_{h,x}(y,z)\|y-z\|\, {\pi_h^x}(\d y,\d z)-\frac{\phi_h\big(\Wc_p(\mu_h,\nu_h^x)\big)}{h}+\frac{\ep}2\\
			 &  \qquad\quad\leq \bigg(\int_{X\times X}\big|g_{h,x}(y,z)\big|^q\, \pi_h^x(\d y,\d z)\bigg)^{1/q}\frac{\Wc_p(\mu_h,\nu_h^x)}{h}-\phi\bigg(e^{-c h}\frac{\Wc_p(\mu_h,\nu_h^x)}{h}\bigg)+\frac{\ep}2\\
			& \qquad \quad \leq \phi^*\Bigg(e^{c h}\bigg(\int_{X\times X}\big|g_{h,x}(y,z)\big|^q\, \pi_h^x(\d y,\d z)\bigg)^{1/q}\Bigg)+\frac{\ep}2
		\end{align*}
		for all $h\in [0,1]$ and $x\in X$, where $q=\frac{p}{p-1}$ is the conjugate exponent of $p$. Since $\phi^*$ is convex and continuous, thus uniformly continuous on bounded intervals, and $u'$ is bounded, there exists some $\delta>0$ such that, for all $x\in X$,
		\[
		 \phi^*\big((1+\delta )(\|u'(x)\|+\delta)\big)\leq \phi^*\big(\|u'(x)\|\big)+\frac{\ep}{2}.
		\]
		Since $\phi^*$ is nondecreasing and $e^{c h}\leq 1+\delta $, for $h>0$ sufficiently small, \eqref{eq.keygenerator3} follows as soon as we are able to show that
		\[
		\sup_{\|x\|\leq r}\Bigg(\bigg(\int_{X\times X}\big|g_{h,x}(y,z)\big|^q\, \pi_h^x(\d y,\d z)\bigg)^{1/q}- \|u'(x)\|\Bigg)\leq \delta
		\]
		for $h>0$ sufficiently small.\ Observe that
		\begin{align}
		\notag g_{h,x}(y,z)&=\frac{1}{\|y-z\|}\bigg|\int_0^1 u'\big(\psi_h(x)+sy+(1-s)z\big)(y-z)\,\d s\bigg|\\
		\label{eq.fundamentalthm}&\leq  \int_0^1 \big\|u'\big(\psi_h(x)+sy+(1-s)z\big)\big\|\,\d s.
		\end{align}		
		Using \eqref{eq.fundamentalthm}, the uniform continuity of $u'$ on bounded subsets of $X$, and Remark \ref{rem.observations0} a), there exists some $\delta'>0$ such that
		\[
		g_{h,x}(y,z)\leq \|u'(x)\|+\frac\delta 2
		\]
		for all $x,y,z\in X$ with $\|x\|\leq r$ and $\|y\|+\|z\|\leq \delta'$, and $h>0$ sufficiently small. Using Minkowski's inequality and Markov's inequality together with \eqref{eq.fundamentalthm},  we may conclude that
		\begin{align*}
		\bigg(\int_{X\times X}\big|g_{h,x}(y,z)\big|^q\, \pi_h^x(\d y,\d z)\bigg)^{1/q}&\leq \|u'(x)\|+\frac\delta 2+\bigg(\frac{\|u'\|_\infty}{\delta'}\Wc_p(\pi_h^x,\delta_0)\bigg)^{1/q}\\
		&\leq\|u'(x)\|+\delta
		\end{align*}
		for all $x\in X$ with $\|x\|\leq r$ and $h>0$ sufficiently small, since, by Minkowski's inequality and our global assumption (A3),
		\begin{align*}
		\sup_{x\in X}\Wc_p(\pi_h^x,\delta_0) &\leq \bigg(\int_X \|y\|^p\,\mu_h(\d y)\bigg)^{1/p}+\sup_{x\in X} \Wc_p(\mu_h,\nu_h^x)+\Wc_p(\mu_h,\delta_0)\\
		&\leq 2\bigg(\int_X \|y\|^p\,\mu_h(\d y)\bigg)^{1/p}+bh^\alpha\to 0 \quad \text{as }h\downarrow 0.
		\end{align*}
		The proof is complete.
	\end{proof}
	
	Before we start with the construction of the distributionally robust version of the family $\big(T(t)\big)_{t\geq 0}$, cf. Remark \ref{rem.observations0} c), we define the fundamental object for the rest of our study. 
	
	\begin{definition}\label{def.semigroup}
		We say that a family $S=\big(S(t)\big)_{t\geq 0}$ is a \textit{strongly continuous convex transition semigroup} (on $\BUC$) if the following conditions hold:
		\begin{enumerate}
			\item[(i)] For all $t\geq 0$, $S(t)\colon \BUC \to \BUC $ is convex and monotone with $S(t)m=m$, for every constant (function) $m\in \R$, and $\|S(t)u\|_\Lip\leq \|u\|_\Lip$ for all $u\in \Lipb$.
			\item[(ii)] For all $s,t\geq 0$, $S(t)S(s) =S(t+s)$.
			\item[(iii)] $S(t)u_n\to S(t)u$ uniformly in $t\in [0,s]$ as $n\to \infty$ for all $s\geq 0$, $(u_n)_{n\in \N}\subset \BUC$, and $u\in \BUC$ with $u_n\to u$ as $n\to \infty$.
			\item[(iv)]  $S(t)u\to u $ as $ t \downarrow 0 $ for all $u\in \BUC$.
		\end{enumerate}
		For a strongly continuous convex transition semigroup $S$, we define the \textit{generator} $A\colon D(A)\subset \BUC \to \BUC$ of $S$ by
		\[
		D(A):=\bigg\{u\in \BUC\colon  \frac{S(h)u-u}{h}\to g\in \BUC\text{ as }h\downarrow 0\bigg\}\quad \text{and}\quad Au:=g
		\]
		for $u\in D(A)$ and $g\in \BUC$ with $\frac{S(h)u-u}{h}\to g$ as $h\downarrow 0$.
	\end{definition}

	\begin{remark}\label{remark.Smu}
		Choosing $ \phi := \infty \cdot\eins_{(0,\infty)} $, the results of this section imply that the family $T=\big(T(t)\big)_{t\geq 0}$, given by	
		\[ \big(T(t)u\big)(x) = \int_{X}u\big(\psi_t(x)+y\big)\,\mu_{t}(\d y)\quad \text{for all }t\geq 0,\; u \in \BUC,\text{ and }x\in X, \]
		is a strongly continuous convex (even linear) transition semigroup. Note that, for all $t\geq 0$, the operator $T(t)=I(t)$ is linear, which shows that the properties (i), (ii), and (iii) of a strongly continuous convex transition semigroup are satisfied. The semigroup property (ii) follows from \eqref{eq.consistency}, see also Remark \ref{rem.observations0} c). Throughout the remainder of this section, we denote by $B\colon D(B)\subset \BUC\to \BUC$ the generator of $T$.
	\end{remark}

	For $n\in \N_0$ and $t\geq0$, we define $I^n(t)\colon \BUC\to \BUC$ via the following construction. For $ u \in \BUC $, we define
	\begin{equation}\label{eq.iteratedyadic}
		I^{n}(t)u := \big(\underbrace{I(2^{-n})\cdots I(2^{-n})}_{k\text{-times}}\big)I\big(t-k2^{-n}\big)u=I(2^{-n})^kI\big(t-k2^{-n}\big)u,
	\end{equation}
	where $k\in \N_0$ denotes the largest natural number with $k2^{-n}\leq t$. Then, by Lemma~\ref{lem.monotonicity},
	\[
	I^{n+1}(t)u\leq I^n(t)u\quad \text{for all }n\in \N_0,\; t\geq0,\text{ and }u\in \BUC.
	\]
	We define
	\begin{equation}\label{eq.defst}
	S(t)u:=\inf_{n\in \N_0}I^n(t)u\quad \text{for all }t\geq 0\text{ and }u\in \BUC.
	\end{equation}
	Note that $S(t)u=\lim_{n\to \infty} I^n(t)u$, since $I^{n+1}(t)u\leq I^n(t)u$, for all $n\in \N$, $t\geq 0$, and $u\in \BUC$. The previous results allow us to state the following main result of this section.
	
	\begin{theorem}\label{thm.semigroup}
		The family $S$ is a strongly continuous convex transition semigroup. Let $A$ denote the generator of $S$. Then, $D(B)\cap \BUC^1\subset D(A)$ with
		\[
		\big(Au\big)(x)=\big(B u\big)(x)+\phi^*\big(\|u'(x)\|\big)\quad \text{for all }u\in D(B)\cap \BUC^1\text{ and }x\in X.
		\]
	\end{theorem}

	\begin{proof}
		Let $t\geq 0$. Then, for $n\in \N$ sufficiently large, $I^n(t)u\in \BUC$ for all $u\in \BUC$ since, by definition, $I^n(t)$ is a finite composition of operators $I(h)\colon \BUC\to \BUC$ with $h\in (0, h_0]$, cf.\ Lemma~\ref{lem.keyestimates}. Hence, for $n\in \N$ sufficiently large, the operator $I^n(t)\colon \BUC\to \BUC$ is well-defined, convex, and monotone with $I^n(t)m=m$ for all $m\in \R$ and $\|I^n(t)u\|_\Lip\leq e^{ct}\|u\|_\Lip$ for all $u\in \Lipb$ as all these properties carry over from $I(h)$, for $h>0$, to $I^n(t)$. Since $\|I^n(t)u\|_\Lip\leq e^{ct}\|u\|_\Lip$ for all $u\in \Lipb$ and
		\begin{equation}\label{eq.proofmain000}
		 T(t)u\leq S(t)u\leq I^n(t)u\quad\mbox{for all }n\in \N\mbox{ and }u\in \BUC,
		\end{equation}
                it follows that $S(t)u\in \Lipb$ with $\|S(t)u\|_\Lip\leq e^{ct}\|u\|_\Lip$ for all $u\in \Lipb$. Next, we verify property (iii) in Definition \ref{def.semigroup}. For this, we even show a slightly stronger property. For $i=1,2$, let $(u_k^i)_{k\in \N}\subset \BUC$ with $(u_k^1-u_k^2)\to 0$ as $k\to \infty$, i.e.,
	       \[
	        \sup_{n\in \N} \|u_k^1-u_k^2\|_\infty<\infty\quad \text{and}\quad \lim_{k\to \infty}\sup_{\|x\|\leq r}\big|u_k^1(x)-u_k^2(x)\big|\to 0\quad \text{for all }r\geq 0.
	       \]
	       We prove that, for $s\geq 0$ and $r\geq 0$,
        	 \begin{equation}\label{eq.convS}
        	 \lim_{k\to \infty} \sup_{t\in [0,s]}\sup_{\|x\|\leq r} \big|\big(S(t)u_k^1\big)(x)-\big(S(t)u_k^2\big)(x)\big|=0.
	        \end{equation}
                That is, $\big(S(t)u_k^1-S(t)u_k^2\big)\to 0$ uniformly in $t\in [0,s]$ as $k\to \infty$ for all $s\geq 0$. To that end, observe that, for $t\geq 0$, $\lambda\in (0,1)$, and $n\in \N$,\footnote{The first inequality of the following estimate follows from the convexity of the mapping $z\mapsto S(t)(u_k^2+z)-S(t)u_k^2$ for $z_1=\frac{u_k^1-u_k^2}{\lambda}$ and $z_2=0$.}
                \begin{align*}
                 S(t)u_k^1-S(t)u_k^2&\leq \lambda\Big(S(t)\Big(u_k^2+\tfrac{u_k^1-u_k^2}{\lambda}\Big)-S(t)u_k^2\Big)\leq \lambda\Big(I^n(t)\Big(u_k^2+\tfrac{u_k^1-u_k^2}{\lambda}\Big)-S(t)u_k^2\Big)\\
                 &\leq \lambda\Big(I^n(t)\Big(u_k^2+\tfrac{u_k^1-u_k^2}{\lambda}\Big)-I^n(t)u_k^2\Big) +2\lambda \|u_k^2\|_\infty.
                \end{align*}
                The statement now follows from a symmetry argument together with Lemma~\ref{lem.technical} and an appropriate choice of $\lambda\in (0,1)$ (sufficiently small), $n\in \N$ (sufficiently large in order to achieve $2^{-n}\leq h_0$ in Lemma~\ref{lem.technical}), and $k\in \N$ (sufficiently large). Approximating $u\in \BUC$ with a sequence $(u_k)_{n\in \N}\subset \Lipb$, it follows that $S(t)\colon \BUC\to \BUC$ is well-defined. Since all other properties stated in (i) of Definition \ref{def.semigroup} directly carry over from $I^n(t)$ to the limit, $S(t)$ satisfies these properties. Moreover, \eqref{eq.proofmain000} together with Lemma~\ref{lem.keyestimates} and Remark \ref{remark.Smu} implies that $S(t)u\to u$ as $t\downarrow 0$ for all $u\in \BUC$.\\
                In order to verify the semigroup property (ii), we fix $s,t\geq 0$ and $u\in \Lipb$. Consider the set $\mathcal D:=\{k2^{-n} \colon  k,n\in \N_0\}$ of all dyadic numbers. We first show the semigroup property in the case, where $t$ is a dyadic number, i.e., $t\in \mathcal D$. Then, by definition of $S$, 
                \[ S(t+s)u = \lim_{n \to \infty}I^n(t+s)u=\lim_{n \to \infty}I^n(t)I^n(s)u, \]
                where, in the second equality, we used the fact that $t$ is a dyadic number. Due to the monotonicity of $S$, $I^n(t)I^n(s)u\geq S(t)S(s)u$ for all $n\in \N$. On the other hand, for fixed $k\in \N$,
                \[
                 \lim_{n \to \infty}I^n(t)I^n(s)u\leq \lim_{n \to \infty}I^n(t)I^k(s)u=S(t)I^k(s)u.
                \]
                Now, since $I^k(s)u$, for $k\in \N$, is a decreasing sequence of monotone functions in $\Lipb$, it follows that $I^k(s)u\to S(s)u$. In fact, since $I^k(s)u\in \Lipb$, there exist bounded and Lipschitz continuous functions $u^k\colon X\to \R$ such that $\big(I^k(s)u\big)(x)=u^k(Kx)$ for all $x\in X$ and $k\in \N$. Again, since $S(t)\in \Lipb$, there exists a bounded and Lipschitz continuous function $u_0\colon X\to \R$. By Dini's lemma, it follows that
                \[
                 \lim_{k\to \infty}\sup_{\|x\|\leq r}\big|I^k(s)u\big)(x)-\big(S(s)u\big)(x)\big|=\lim_{k\to \infty}\sup_{\|x\|\leq r}|u^k(Kx)-u_0(Kx)|= 0
                \]
                for all $r\geq 0$, where we used the fact that $K$ is a compact operator.  We have therefore shown that $S(t+s)u=S(t)S(s)u$, when $t$ is a dyadic number and $u\in \Lipb$. For the general case, we approximate $t\geq 0$ with dyadic numbers $(t_n)_{n\in \N}\subset \mathcal D$ and $u\in \BUC$ with functions $(u_n)_{n\in \N}\subset \Lipb$, and obtain, using the properties (iii) and (iv),
                \[
                 S(t+s)u = \lim_{n \to \infty} S(t_n+s)u_n = \lim_{n \to \infty}S(t_n)S(s)u_n = S(t)S(s)u.
                \]
		Now, let $u\in D(B)\cap \BUC^1$ and $g(x):=\big(Bu\big)(x)+\phi^*\big(\|u'(x)\|\big)$ for all $x\in X$. By definition of the generator $B$ and Lemma~\ref{lem.keygenerator}, we find that $g\in \BUC$. Let $a\geq 0$ and $I_a(t)$, for $t\geq 0$, be given as in Lemma~\ref{lemma1} with $L:=\|u\|_\Lip$. Let $h\in [0,1]$ and $0=t_0<\ldots<t_m=h$ be a partition of the interval $[0,h]$ with $m\in \N$ and $h_k:=t_k-t_{k-1}$ for $k\in \{1,\ldots, m\}$. Then, by Lemma~\ref{lemma1} and Lemma ~\ref{lem.monotonicity} for the family of operators $\big(I_a(t)\big)_{t\geq 0}$,
		\begin{align*}
		 \frac{I(h_1)\cdots I(h_{k-1})u-I(h_1)\cdots I(h_k)u}{h_k}+g&\leq I_a(h_1)\cdots I_a(h_{k-1})\bigg(\frac{u-I(h_k)u}{h_k}\bigg)+g\\
		 &\leq I_a(t_{k-1})\bigg(\frac{u-I(h_k)u}{h_k}\bigg)+g.
		\end{align*}
                We thus obtain that
		\begin{align*}
		\frac{u-I(h_1)\cdots I(h_m)u}{h}+g&= \sum_{k=1}^m\frac{h_k}{h}\bigg(\frac{I(h_1)\cdots I(h_{k-1})u-I(h_1)\cdots I(h_k)u}{h_k}+g\bigg)\\
		& \leq \sum_{k=1}^m\frac{h_k}{h}\Bigg(I_a(t_{k-1})\bigg(\frac{u-I(h_k)u}{h_k}\bigg)+g\Bigg).
		\end{align*}
	 Taking the supremum over all finite partitions of the interval $[0,h]$, it follows that
	 \[
	  \frac{u-S(h)u}{h}+g\leq \sup_{t,h'\in (0,h]}\frac{h'}{h}\Bigg(I_a(t)\bigg(\frac{u-I(h')u}{h'}\bigg)+g\Bigg).
	 \]
                Using the convergence results for the family of operators $\big(I_a(t)\big)_{t\geq 0}$ from Lemma~\ref{lem.technical} and Lemma~\ref{lem.keyestimates} together with Lemma~\ref{lem.keygenerator}, we observe that the right-hand side converges to zero as $h\downarrow 0$. This together with the observation
		\[
		\frac{S(h)u-u}{h}-g\leq \frac{I(h)u-u}{h}-g,
		\]
                yields that $u\in D(A)$ with $Au=g$.
	\end{proof}
	
	\begin{remark}\label{rem:Ipi}
		For $t>0$, let $P_t$ denote the set of all partitions $\pi=\{t_0,\ldots, t_m\}$ with $0=t_0<\ldots <t_m=t$ and $m\in \N$. For $t>0$, $\pi\in P_t$, and $u\in \BUC$, we define
		\[
		I(\pi) u:=I({t_1-t_0})\cdots I({t_m-t_{m-1}})u.
		\]
		Then, for $t>0$ and $u\in \BUC$,
		\[
		S(t)u=\inf_{\pi\in P_t} I(\pi) u.
		\]
		Indeed, by definition of $S(t)$, the inequality $S(t)u\geq \inf_{\pi\in P_t} I(\pi) u$ holds true.\ In order to establish the other inequality, let $\pi\in P_t$. Then, using the semigroup property of $S$,
		\[
		I(\pi) u\geq S({t_1-t_0})\cdots S({t_m-t_{m-1}})u=S(t)u\quad \text{for all }u\in \BUC.
		\]
		
	\end{remark}

	\section{Extensions and Remarks}\label{sec.remarks}

	\subsection{On the particular form of the transition kernels}\label{sec.discussion}
	In this section, we discuss the limitations of our setup in terms of the structural assumption \eqref{eq.markovdecomp}, i.e.,
%	\[
%	p_t(x,B):=\mu_t\big(\{y\in X\colon\psi_t(x)+y\in B\}\big),\quad\mbox{for }t\geq 0,\; x\in X,\text{ and }B\in\Bc(X),
%	\]
	\[
     \Xi_t^x=\psi_t(x)+Y_t,
    \]
	on the reference Markov process $(\Xi_t^x)$.  In the following, we explore how this class behaves regarding
	\begin{enumerate}
		\item[a)] deterministic dynamics,
		\item[b)] stochastic dynamics arising from affine processes,
		\item[c)] the structure of the underlying Banach space and vector operations.
	\end{enumerate}	
	\subsubsection*{a) Deterministic dynamics}
	In this part, we consider the case of deterministic dynamics, i.e., $Y_t\equiv0$ for all $t\geq 0$ in \eqref{eq.markovdecomp}. Then, the consistency condition \eqref{eq.consistency}, simplifies to the flow property of the family $(\psi_t)_{t\geq 0}$, i.e.,
	\[
	\psi_t\circ \psi_s=\psi_{t+s}\quad \text{for all }s,t\geq 0.
	\]
	Since $\psi_0(x)=x$ for all $x\in X$, this means that $(\psi_t)_{t\geq 0}$ is a so-called semiflow, and the reference semigroup is the so-called Koopman semigroup related to $(\psi_t)_{t\geq 0}$.\  In particular, the dynamics comprise all solutions to stationary ordinary differential equations with a global Lipschitz condition as well as solutions to linear partial differential equations as we discuss in Example \ref{sec.koop} and Example \ref{sec.lindyn}, respectively.
	\subsubsection*{b) Stochastic dynamics arising from affine processes}
	In stochastic modeling, random evolutions are oftentimes implemented using stochastic differential equations (SDEs). In this case, structural assumptions are imposed on the coefficients appearing in the SDE. A widely used assumption is the one that the It\^o diffusion is an affine process, i.e., up to technicalities, the coefficients appearing in the SDE are affine linear functions depending on the current state of the process, cf.\ Filipovi\'{c} \cite{MR2553163}.\ In this section, we investigate how the class of Markov processes that admit a decomposition of the form \eqref{eq.markovdecomp} relates to the class of affine processes.
	
	Assume that $(\Xi_t^x)$ is an $\R^d$-valued stochastic process, which is a solution to the SDE
	\[
	\d \Xi_t^x=b(\Xi_t^x)\, \d t+\rho(\Xi_t^x)\,\d W_t,\quad \Xi_0^x=x\in \R^d,
	\]
	where $(W_t)_{t\geq 0}$ is a $d$-dimensional Brownian motion on a suitably filtered probability space $(\Omega,\mathcal F,(\mathcal F_t)_{t\geq 0},\P)$ satisfying the usual assumptions, $b\colon \R^d\to \R^d$ is continuous, and $\rho\colon \R^d\to \R^{d\times d}$ is continuous and takes values in the set of all symmetric positive semidefinite matrices. Assume that $(\Xi_t^x)$ is an affine process in the sense that $b$ and $a:=\rho^T\rho$ are affine functions, i.e.,
	%\begin{align}
	%	\label{eq.affa} a(x)&=\alpha_0+\sum_{i=1}^dx_i\alpha_i\quad\text{and}\\
	%	\label{eq.affb} b(x)&=\beta_0 +\sum_{i=1}^dx_i\beta_i=b+\beta x
	%\end{align}
	\[
	 a(x)=\alpha_0+\sum_{i=1}^dx_i\alpha_i\quad\text{and}\quad  b(x)=\beta_0 +\sum_{i=1}^dx_i\beta_i=\beta_0+\beta x
	\]
	with some matrices $\alpha_0,\ldots,\alpha_d\in \R^{d\times d}$, $\beta_0,\ldots,\beta_d\in \R^d$ and $\beta=(\beta_1,\ldots,\beta_d)^T$, and that, for all $u\in i\R^d$, the Riccati equations
	\begin{align}
		\label{eq.Ric1} \partial_t \Phi(t,u)&=\frac12 \Psi(t,u)^T\alpha_0\Psi(t,u)+\beta_0^T \Psi(t,u),\quad \Phi(0,u)=0,\\
		\label{eq.Ric2} \partial_t	\Psi_i(t,u)&=\frac12 \Psi(t,u)^T\alpha_i\Psi(t,u)+\beta_i^T \Psi(t,u),\quad \Psi(0,u)=u,
	\end{align} 
	have a solution $(\Phi,\Psi)$ with $\Psi=(\Psi_1,\ldots, \Psi_d)$ and ${\rm Re}\, \big(\Phi(t,u)+\Psi(t,u)^T x\big)\leq 0$. Then, the characteristic function $\phi_\Xi$ of $(\Xi_t^x)$ is of the form
	\[
	\qquad\phi_\Xi(t,x,u):=\E\big(e^{u^T \Xi_t^x}\big)=e^{\Phi(t,u)+ \Psi(t,u)^T x}\quad \text{for all }t\geq 0,\; x\in \R^d,\text{ and }u\in i\R^d.
	\]
	In this case, our structural assumption that $(\Xi_t^x)$ is of the form \eqref{eq.markovdecomp}, implies that
	\begin{equation}\label{eq.aff}
	e^{\Phi(t,u)+ \Psi(t,u)^T x}=\E\big(e^{u^T \Xi_t^x}\big)=e^{u^T \psi_t(x)} \E\big( e^{u^T Y_t}\big)
	\end{equation}
	for all $t\geq 0$, $x\in \R^d$, and $u\in i\R^d$. For $x=0$, this yields that $	e^{\Phi(t,u)}=e^{u^T\psi_t(0)} \E\big( e^{u^T Y_t}\big)$
	%\[
	%e^{\Phi(t,u)}=e^{u^T\psi_t(0)} \E\big( e^{u^T Y_t}\big)
	%\]
	for all $t\geq 0$ and $u\in i\R^d$, which together with \eqref{eq.aff} leads to
	\[
	\Psi(t,u)^T x=u^T \big(\psi_t(x)-\psi_t(0)\big)\quad \text{for all }t\geq 0,\; x\in \R^d,\text{ and }u\in i\R^d.
	\]
	Hence, for all $t\geq 0$, the maps $u\mapsto \Psi(t,u)$ and $x\mapsto \psi_t(x)-\psi_t(0)$ are linear in $u$ and $x$, respectively. Since $(\Phi,\Psi)$ solve the Riccati equations \eqref{eq.Ric1} and \eqref{eq.Ric2}, it follows that $\alpha_i=0$ for $i=1,\ldots,d$ and
	$$\Psi(t,u)=e^{t \beta}u\quad \text{for all }t\geq 0\text{ and }u\in i\R^d.$$
	As a consequence,
	$$
	\d \Xi_t^x=\big(\beta_0+\beta X_t^x\big)\, \d t+\sqrt{\alpha_0}\,\d W_t,\quad \Xi_0^x=x\in \R^d,
	$$
	and
	$$\psi_t(x)=e^{t \beta}x+\int_0^te^{\beta s}\beta_0\, \d s\quad \text{for all }t\geq 0\text{ and }x\in \R^d.$$
	We conclude that every affine process $(\Xi_t^x)$, which satisfies \eqref{eq.markovdecomp}, is, up to a (e.g, logarithmic) transformation, see part c), an Ornstein-Uhlenbeck process.
	
   \subsubsection*{c) Structure of the underlying Banach space and vector operations}
	%Let $\psi_t\colon X\to X$ be a map with
%\[
%\big\|V\big(\psi_t(x_1)\big)-V\big(\psi_t(x_2)\big)\big\|\leq \|V(x_1)-V(x_2)\| \quad \text{for all }t\geq 0\text{ and }x_1,x_2\in X. 
%\]
 In this part, we show how different vector operations on an a priori arbitrary nonempty set $X$ can be used to incororate more general versions of the additive operation $+$ in \eqref{eq.markovdecomp}, for instance a multiplication. For this purpose, let $X$ be a nonempty set, and assume that there exists a bijective function $V\colon X\to M$, where $M$ is a separable Banach space.
	Then, for $x,y\in X$ and $\la\in \R$, the operations
	\[
	x+_Vy:=V^{-1}\big(V(x)+V(y)\big),\quad \lambda \cdot_Vx :=V^{-1}\big(\la V(x)\big),\quad \text{and}\quad \|x\|_V:=\|V(x)\|
	\]
	define an addition, a scalar multiplication, and a norm on $X$, respectively. By definition, the map $V\colon X\to M$ is an isometric isomorphism making $X$ a separable Banach space, as well. In this setup, one may consider flows of the form
	\[
	\Xi_t^x=\psi_t(x)+_VY_t.
	\]
	For a function $u\in \BUC$ with $u\circ V^{-1}\colon M\to \R$ differentiable and $x\in X$, the derivative $u'(x)\colon X\to \R$ (w.r.t.\ the addition $+_V$ and scalar multiplication $\cdot_V$) is then given by 
	\begin{align}
		\begin{split}\label{eq.derivative}
		u'(x)z&=\lim_{h\downarrow 0} \frac{u\big(x+_V (h\cdot_V z)\big)-u(x)}{h}=\lim_{h\downarrow 0} \frac{u\Big(V^{-1}\big(V(x)+hV(z)\big)\Big)-u(x)}{h}\\
		&=\Big(D\big(u\circ V^{-1}\big)\big(V(x)\big)\Big)V(z)\quad \text{for all }z\in X.
		\end{split}
	\end{align}
	\begin{remark}
        Assume that $X$ is an open subset of some Banach space $X_0$ and that $V\colon X\to M$ is a $C^1$-diffeomorphism from $X$ (as an open subset of $X_0$) to $M$. Then, Equation \eqref{eq.derivative} together with the chain rule implies that, for any continuously differentiable function $u\colon X\to \R$ and $x,z\in X$,
	\begin{equation}\label{eq.Vderivative}
		u'(x)z=\big(Du(x)\big)\big(DV(x)\big)^{-1}V(z),
	\end{equation}
	where $Du$ and $DV$ denote the derivatives of $u$ and $V$ on $X$ as a subset of $X_0$, respectively. That is, for every $x\in X$, $Du(x)$ is an element of the topological dual space of $X_0$ and $\big(DV(x)\big)^{-1}$ is a bounded linear operator from $M$ to $X_0$. In this case,
	\[
	\|u'(x)\|=\sup_{\substack{z\in X\\ V(z)\neq 0}}\frac{|u'(x)z|}{\|V(z)\|}\leq \Big\|\big(Du(x)\big)\big(DV(x)\big)^{-1}\Big\|,
	\]
	where the norm appearing on the left-hand side is the operator norm w.r.t.\ $\|\cdot\|_V$ and the norm on the right-hand side is the operator norm on the dual space of $M$. 	 
	\end{remark}

%        Choosing $V=\log$, our setup covers, for example, the case of geometric dynamics, and the term in the penalisation closely relates to the first order term appearing in the generator of a geometric Brownian motion.
  In the following example, we discuss the previously introduced setup for logarithmic tansformations.
	\begin{example}\label{ex.GBM}
		Consider the particular choices $X=(0,\infty)$, $M=\R$, and $V(x)=\log x$. Then,
		\[
		x+_Vy=xy,\quad \lambda\cdot_V x=x^\lambda,\quad \text{and}\quad \|x\|_V=|\log x|.
		\]
		This leads to dynamics of the form
		\[
		\Xi_t^x=\psi_t(x)Y_t,
		\]
		see Example \ref{ex.geombm} for the case of a geometric Brownian motion. %where $Y_t$ corresponds, for example, to a geometric Brownian motion starting in $1$, i.e., $Y_t=\exp\big(\big(\mu-\tfrac{\sigma^2}{2}\big)+\sigma W_t\big)$ with $\mu\in \R$, $\sigma\geq 0$, and a standard Brownian motion $(W_t)_{t\geq 0}$ defined on some probability space $(\Omega,\mathcal F,\P)$.
		Then, for any continuously differentiable function $u\in \BUC$, the derivative $u'$ is given by
		\[
		u'(x):=u'(x)e=\lim_{h\downarrow 0} \frac{u\big(x+_V (h\cdot_V e)\big)-u(x)}{h}=\lim_{h\downarrow 0} \frac{u(xe^h)-u(x)}{h}=x\frac{\d}{\d x}u(x),
		\]
		where $e$ is the Euler constant and $\frac{\d}{\d x}u$ denotes the usual derivative of $u$.\ Note that the equality $u'(x)=x\frac{\d}{\d x}u(x)$ also follows from \eqref{eq.Vderivative} with $z=e$, since $\big(\frac{\d}{\d x}\log(x)\big)^{-1}=x$ for all $x\in (0,\infty)$. In particular, $\|u'(x)\|=\big|x\frac{\d}{\d x}u(x)\big|$ for all $x\in (0,\infty)$.
	\end{example}

	\subsection{Viscosity solutions}\label{sec.viscosity}
	
	Note that Theorem \ref{thm.semigroup} shows that the family $S$ is a semigroup, which is, at least formally, closely related to a dynamic programming principle. Moreover, we have shown that the local behaviour at time $t=0$, given in terms of the generator $A$, is given by
	\[
	\big(Au\big)(x)=\big(Bu\big)(x)+\phi^ *\big(\|u'(x)\|\big)\quad \text{for }u\in D(B)\cap \BUC^1\text{ and }x\in X,
	\]
	where $B$ is the generator of the reference semigroup $T$ related to the transition kernels $(p_t)_{t\geq 0}$. In view of these results, it is a natural question, whether the semigroup $S$ gives rise to viscosity solutions to the abstract HJB-type differential equation
	\begin{equation}\label{eq.viscositysol}
		\partial_tu=Au\quad \text{for }t\geq 0.
	\end{equation}
	Using a standard procedure from Denk et al.\ \cite{denk2020semigroup}, Nendel and R\"ockner \cite{roecknen}, or Bartl et al.\ \cite{bartl2020limits}, which can be almost literally adapted to our setup, one can show that the map $t\mapsto S(t)u_0$, for fixed $u_0\in \BUC$, is a viscosity solution to \eqref{eq.viscositysol} using the following notion of a viscosity solution. 
	\begin{definition}
		We say that a function $u\colon [0,\infty)\to \BUC$ is a \textit{viscosity subsolution} to the abstract differential equation \eqref{eq.viscositysol} if $u\colon [0,\infty)\to \BUC$ is continuous and
		\[
		\big(\psi'(t)\big)(x)\leq \big(A \psi(t)\big)(x)
		\]
		for all $t>0$, $x\in M$, and every differentiable function $\psi\colon (0,\infty)\to \BUC$ satisfying $\psi(t)\in D(B)\cap \BUC^1$, $\big(\psi(t)\big)(x)=\big(u(t)\big)(x)$, and $\psi(s)\geq u(s)$ for all $s>0$. A function $u\colon [0,\infty)\to \BUC$ is called a \textit{viscosity supersolution} to \eqref{eq.viscositysol} if $-u$ is a viscosity subsolution to \eqref{eq.viscositysol} with $A$ replaced by $-A(-\,\cdot\,)$. A function $u\colon [0,\infty)\to \BUC$ is called a \textit{viscosity solution} to \eqref{eq.viscositysol} if it is both a viscosity subsolution and a viscosity supersolution to \eqref{eq.viscositysol}.
	\end{definition}
	
	We obtain the following proposition.
	\begin{proposition}\label{thm.viscosity}
		Let $u_0\in \BUC$. Then, the function $u\colon [0,\infty)\to \BUC,\; t\mapsto S(t)u_0$ is a viscosity solution to the abstract differential equation \eqref{eq.viscositysol} with $u(0)=u_0$.
	\end{proposition}
	
	The proof is up to the considered function space (here $\BUC$) almost literally the same as in Denk et al.\ \cite[Proposition 5.11]{denk2020semigroup}, Nendel and R\"ockner \cite[Theorem 4.5]{roecknen}, or Bartl et al.\ \cite[Theorem 2.12]{bartl2020limits}.
	
	\subsection{Unbounded functions and additional continuity properties}\label{sec.unbdd}
	Although we consider the space $\BUC$ resulting as a closure from bounded and suitably Lipschitz continuous functions, the construction of the semigroup and many statements carry over to the space $\Lip$ of all functions $u\colon X\to \R$ satisfying \eqref{eq.lipcondK}. For such functions, $\|\cdot\|_\Lip$ is defined as before. The key ingredient in order to establish the semigroup property of $\big(S(t)\big)_{t\geq 0}$ on $\Lip$ is an additional continuity property of $S(t)$ for all $t\geq 0$.
	We start with the following observations.
	
	\begin{remark}\label{rem.unbdd}\
	Let $u\in \Lip$ and $t\geq 0$. We consider $T(t)u$, $I(t)u$, and $S(t)u$ given by \eqref{Smu}, \eqref{eq.isometry}, and \eqref{eq.defst}, respectively. Then, 
	\begin{equation}\label{eq.lingrowth}
	|u(x)|\leq C\big(1+\|x\|\big)\quad \text{for all }x\in X,
	\end{equation}
	with $C:=\max\big\{|u(0)|,\|u\|_\Lip\big\}$, which, by Assumption (A3), implies that
	\[
	\big|\big(T(t)u\big)(x)\big|\leq C\bigg(1+\|\psi_t(x)\|+\int_X \|y\|\, \mu_t(\d y)\bigg)<\infty. 
	\]
	Moreover, by Remark~\ref{rem.centerpoint}, it follows that 
	\[
	I(t)u\leq T(t)u+t\phi^*\big(e^{ct}\|u\|_\Lip\big)
	\]
	for all $u\in \Lip$. We thus end up with the estimate
	\begin{equation}\label{eq:lipest}
	T(t)u\leq S(t)u\leq I(t)u\leq T(t)u+t\phi^*\big(e^{ct}\|u\|_\Lip\big),
	\end{equation}
	which shows that the functions $I(t)u$ and $S(t)u$ are well-defined. Moreover, by similar arguments as in Remark~\ref{rem.basic} b) and  
	Theorem~\ref{thm.semigroup}, we obtain 
	$$\|I(t)u\|_\Lip\leq e^{ct}\|u\|_\Lip\quad\text{and}\quad\|S(t)u\|_\Lip\leq e^{ct} \|u\|_\Lip,$$ which shows that $I(t)u$ and $S(t)u$ are elements of $\Lip$. Note that, for all $r\geq 0$,
	  \begin{align*}
	   \sup_{\|x\|\leq r} \big|\big(T(t)u\big)(x)-u(x)\big|&\leq \|u\|_\Lip\bigg(\sup_{\|x\|\leq r}\|\psi_t(x)-x\|+\int_X \|y\|\,\mu_t(\d y)\bigg)\to 0
	  \end{align*}
         as $t\downarrow 0$. This observation together with \eqref{eq:lipest} implies for all $r\geq 0$,
         \[
	   \sup_{\|x\|\leq r} \big|\big(S(t)u\big)(x)-u(x)\big|\to 0\quad \text{as }t\downarrow 0.
	  \]
	\end{remark}
  This version of strong continuity motivates the following convergence on $\Lip$.
	\begin{notation}
	 For $(u_k)_{k\in \N}$ in $\Lip$ and $u\in \Lip$, we write $u_k \rightrightarrows u$ as $n\to \infty$ if
	  \[
	   \sup_{k\in \N}\|u_k\|_\Lip<\infty, \quad\text{and}\quad \lim_{k\to \infty}\sup_{\|x\|\leq r}\big|u(x)-u_k(x)\big|= 0\quad \text{for all }r\geq 0.
	  \]
	\end{notation}
   
         \begin{proposition}
	  Let $s\geq 0$, $(u_k)_{k\in \N}$ in $\Lip$, and $u\in \Lip$ with $u_k\rightrightarrows u$ as $k\to \infty$. Then, $S(t)u_k\rightrightarrows S(t)u$ uniformly in $t\in [0,s]$ as $k\to \infty$.
	  Moreover,
	   \[
	  S(t+s)u=S(t)S(s)u\quad \text{for all }s,t\geq 0\text{ and }u\in \Lip.
	  \]
	 \end{proposition}

	\begin{proof}
	 Let $s\ge 0$. Similar to Lemma~\ref{lemma1} and the argumentation in Theorem~\ref{thm.semigroup}, since $L:=\sup_{k\in \N}\|u_k\|_\Lip<\infty$, there exists $a\ge 0$ such that 
	 \[
	  |S(t)u_k-S(t)u|\leq I_a^n(t)|u_k-u|\quad \text{for all }t\in [0,s]\text{ and }k,n\in \N.
	 \]
        Therefore, it is sufficient to consider the case, where $u=0$ and $u_k\geq 0$ for all $k\in \N$, and it remains to show that $I_a(h)u_k\rightrightarrows 0$ as $k\to \infty$ uniformly in $[0,h_0]$ with $h_0>0$ sufficiently small. Let $h_0>0$ sufficiently small such that $$\sup_{h\in [0,h_0]}\|\psi_h(0)\|<\infty\quad \text{and}\quad \sup_{h\in [0,h_0]}\int_X \|y\|^p\, \mu_h(\d y)\leq 1,$$
        where we used our global assumptions (A2) and (A3). Then, by Assumption (A1),
        \begin{equation}\label{eq.lipconv0}
         M_r:=\sup_{h\in [0,h_0]}\sup_{\|x\|\leq r}\|\psi_h(x)\|\leq \sup_{h\in [0,h_0]}\|\psi_h(0)\|+e^{ch_0}r<\infty \quad \text{for all }r\geq 0.
        \end{equation}
        Let $\ep>0$ and $M>0$ with $LM^{1-p}(1+ah_0)^p<\ep$. Then,
        \begin{align*}
         0&\leq \sup_{h\in [0,h_0]}\sup_{\|x\|\leq r}\big(I_a(h)u_k\big)(x)\\
         &\leq  \sup_{\|\xi\|\leq M+M_r} |u_k(\xi)| +\sup_{h\in [0,h_0]}\sup_{\|x\|\leq r}\sup_{\Wc_p(\mu_h,\nu)\leq ah}\int_{\{\|z\|>M\}}\big(u_k(\psi_h(x))+L\|z\|\big)\,\nu(\d z)\\
         &\leq 2\sup_{\|\xi\|\leq M+M_r} |u_k(\xi)| + \sup_{h\in [0,h_0]}\sup_{\Wc_p(\mu_h,\nu)\leq ah} LM^{1-p}\int_{\{\|z\|>M\}}\|z\|^p\,\nu(\d z)\\
         & \leq 2\sup_{\|\xi\|\leq M+M_r} |u_k(\xi)| + LM^{1-p}\sup_{h\in [0,h_0]}\big(\Wc_p(\mu_h,\delta_0)+ah\big)^p\\
         &\leq 2\sup_{\|\xi\|\leq M+M_r} |u_k(\xi)| + LM^{1-p}(1+ah_0)^p < \ep
        \end{align*}
        for $k\in \N$ sufficiently large.
        
        Finally, we show the semigroup property. To do so, for all $k\in \N$, let $\rho_k\colon \R\to \R$ be $1$-Lipschitz with $0\leq \rho_k\leq 1$, $\rho_k(w)=1$ for all $w\in \R$ with $|w|\leq k$, and $\rho_k(w)=0$ for all $w\in \R$ with $|w|\geq k+1$. For $u\in \Lip$, let $u_k(x):=u(x)\rho_k(\|Kx\|)$ for all $x\in X$ and $k\in \N$, so that $(u_k)_{k\in\N}\subset \Lip_b$
        and  $u_k\rightrightarrows u$ as $k\to \infty$. Hence, it follows from the first part and Theorem~\ref{thm.semigroup} that 
        \[
        S(t+s)u=\lim_{k\to \infty} S(t+s)u_k=\lim_{k\to \infty} S(t)S(s)u_k=S(t)S(s)u.\qedhere
        \]
	\end{proof}

	\subsection{Sensitivity analysis for the robust semigroup}\label{sec.sensitivity}
	The sensitivity of functionals depending on the distribution $\mu$ of a random source $X$ with respect to ``small'' nonparametric perturbations of the distribution $\mu$ has recently received a lot of attention.\ In this context, sensitivity is typically understood as a suitable derivative w.r.t.\ the degree of uncertainty expressed in terms of a transport distance, typically, a Wasserstein distance, cf.\ Bartl et al.\ \cite{bartl2020robust}. In this subsection, we address this issue in the context of Wasserstein perturbed semigroups.\ We show that the generator of the semigroup can, in certain situations, be understood as a sensitivity estimate for infinitesimally small amounts of model uncertainty. On the other hand, we address the numerical implementation of the semigroup $S$ using, e.g., Monte-Carlo methods together with the obtained sensitivity bounds.

	Note that the simulation of the semigroup $T$ is of a very simple nature due to our structural assumption on the kernels $(p_t)_{t\geq 0}$. In order to compute a realisation of $\Xi_t^x$ for some $t\geq 0$ and \textit{all} $x\in X$, it is sufficient to perform \textit{one} Monte-Carlo simulation for the random variable $Y_t$ with law $\mu_t$ independent of $x\in X$. In a second step, one simulates $\Xi_t^x$ by computing the sum of the deterministic value $\psi_t(x)$ and the simulated random variable $Y_t$ for all $x\in X$. 
	
	Recall that $T(t)u\leq S(t)u\leq I(t)u$ for all $u\in \Lip$.\ Combined with (\ref{eq.difference}), we have seen that
	\begin{equation}\label{eq.sensitivity1}
		S(t)u-T(t)u\leq I(t)u-T(t)u\leq t\phi^*\big(e^{c t}\|u\|_\Lip\big)\quad \text{for all }u\in \Lip.
	\end{equation}
	Although this estimate is very rough at points where the function $u$ is (almost) constant, it is very attractive due to its simple nature and delivers a reliable estimate for $S$ and $I$ in terms of $T$ that scales linearly in time. Using the modified version of \eqref{eq.sensitivity1} given in Theorem \ref{thm.semigroup}, for all $r\geq 0$,
	\begin{equation}\label{eq.sens0}
	 \sup_{\|x\|\leq r} \bigg|\frac{\big(S(h)u\big)(x)-\big(T(h)u\big)(x)}{h}-\phi^*\big(\|u'(x)\|\big)\bigg|\to 0\quad \text{as }h\downarrow 0,
	\end{equation}
	we also take into account the local behaviour of the function $u\in \BUC^1$. Note that this is a tighter but asymptotic bound for infinitesimally small times $h\geq 0$ that still scales linearly in time.
	
	Implicitly, we have also derived sensitivity bounds for the influence of Wasserstein perturbations around the reference semigroup $T$ in terms of the penalisation $\phi^*$. Consider the case, where $\phi= \infty\cdot \eins_{(a,\infty)}$ with some $a\geq 0$. Note that, for $a=0$, $S=T$. For $a>0$, the operator $I$ is given by
	\[
	 \big(I(t)u\big)(x)=\sup_{\Wc_p(\mu_t,\nu)\leq at} \int_X u\big(\psi_t(x)+z\big)\, \nu(\d z)\quad \text{for }t\geq 0,\, u\in \Lip,\text{ and }x\in X.
	\]
	In this case, the derivative at $t=0$, \eqref{eq.difference1}, reformulates to
	\begin{equation}\label{eq.sens1}
	  \sup_{\|x\|\leq r}\bigg|\frac{\big(I(h)u\big)(x)-\big(T(h)u\big)(x)}{t}-a\|u'(x)\|\bigg|\to 0\quad \text{as }h\downarrow 0,
	\end{equation}
        for $u\in \BUC^1$, and provides a sensitivity estimate for the influence of nonparametric model uncertainty in terms of Wasserstein balls around the baseline model given in terms of the kernels $(p_t)_{t\geq 0}$.\ Moreover, \eqref{eq.sensitivity1} leads to the following version of \eqref{eq.sens1}:
        \[
         \|I(t)u-T(t)u\|_\infty\leq ta\|u\|_\Lip\quad \text{for all }t\geq 0\text{ and }u\in \Lip.
        \]
%The fact that $v(t,x)=\big(S(t)u\big)(x)$, for $t\geq0$, $x\in X$, and $u\in \BUC^1$, defines a viscosity solution to the nonlinear PDE
%\begin{equation}
%	\partial_t v(t,x)=Bv (t,x)+a\|D_xv(t,x)\|\quad \text{for }t\geq 0\text{ and }x\in X,
%\end{equation}
%where $B$ is the generator of $T$ and $D_x$ is the Fr\'echet derivative in the space variable, provides tools for the numerical computation of $S(t)u$ also for large times $t\geq 0$.

         \subsection{Relation to parametric uncertainty and dynamic risk measures}\label{sec:param}
           Closely related to the topic of sensitivity analysis that we discussed in the previous subsection, a lot of interest has recently been paid to reducing optimisation procedures under nonparametric uncertainty to low-dimensional optimisation problems, cf.\ Bartl et al.~\cite{bartl2020computational}, Mohajerin Esfahani and Kuhn~\cite{esfahani2018data}, Zhao and Guan~\cite{MR3771309}, and Blanchet and Murthy~\cite{blanchet2019quantifying}.\ In this subsection, we show how, in special yet relevant cases, the nonparametric uncertainty in the semigroup $S$ reduces to a simple form of parametric uncertainty. This leads to a low-dimensional optimisation scheme instead of the infinite-dimensional optimisation problem related to $I(t)$. More precisely, we discuss the relation between the Wasserstein perturbed semigroup $S$ and the Nisio semigroup, cf.\ \cite{MR0451420}, w.r.t.\ drift uncertainty of the related stochastic process $(\Xi_t^x)$.

         Throughout this subsection, we work under the assumption that $\psi_t(x)=x$ for all $t\geq 0$ and $x\in X$. For $t\geq 0$, $u\in \BUC$, and $x\in X$, we consider
         \begin{equation}\label{eq.defet}
          \big(E(t)u\big)(x):=\sup_{\theta\in X}\bigg(\int_X u\big(x+y+t\theta\big)\, \mu_t(\d y)-t\phi \big(e^{-c t}\|\theta\|\big)\bigg).
         \end{equation}
         Since $\Wc_p(\mu_t,\mu_t\ast \delta_{t\theta})\leq t\|\theta\|$ for all $t\geq 0$ and $\theta\in X$, it follows that
         \begin{equation}\label{eq:Et}
          T(t)u\leq E(t)u\leq I(t)u\quad \text{for all }t\geq 0.
         \end{equation}
         Moreover, one readily verifies that $E(t)\colon \Lipb\to \Lipb$ with $$\|E(t)u\|_\Lip\leq e^{c t} \|u\|_\Lip\quad\text{and}\quad\|E(t)u\|_\infty\leq \|u\|_\infty\quad \text{for all }u\in \Lipb.$$ Similar as in the proof of Lemma~\ref{lemma1}, one sees that also $E(h)\colon \BUC\to \BUC$ is well-defined and $1$-Lipschitz for sufficiently small $h\geq 0$. For $u\in \BUC$ and $\pi=\{t_0,\ldots, t_m\}\in P_t$ with $m\in \N$ and $0=t_0<\ldots <t_m=t$ (see Remark \ref{rem:Ipi}), we define
         \[
          E(\pi)u:=E(t_1-t_0)\cdots E(t_m-t_{m-1})u.
         \]
         Let $u\in \BUC$. Then, \eqref{eq:Et} implies that
         \begin{equation}\label{eq.param0}
          E(\pi)u\leq I(\pi)u.
         \end{equation}
         For all $n\in \N$, we consider the partition $\pi_n:=\{j2^{-n}\colon j=0,\ldots k\}\cup \{t\}\in P_t$ with $k\in \N_0$ being the largest natural number such that $k2^{-n}\leq t$. Then, $$E(t)u\leq E(\pi_n)u\leq E(\pi_{n+1})u\quad\text{and}\quad  I(\pi_n)u=I^n(t)u$$
         for all $n\in \N$. Hence, using \eqref{eq.param0},
         \[
          E(t)u\leq \sup_{n\in \N}E(\pi_n)u=\lim_{n\to \infty}E(\pi_n)u\leq \lim_{n\to \infty}I^n(t)u=S(t)u.
         \]
          Now, let $\pi=\{t_0,\ldots, t_m\}\in P_t$ with $m\in \N$ and $0=t_0<\ldots <t_m$. Then,
         \[
          E(\pi)u= E(t_1-t_0)\cdots E(t_m-t_{m-1})u\leq S(t_1-t_0)\cdots S(t_m-t_{m-1})u=S(t)u
         \]
         This shows that
         \[
          N(t)u:=\sup_{\pi\in P_t}E(\pi)u\leq S(t)u
         \]
         for all $t\geq 0$ and $u\in \BUC$.\ Proceeding as in Nendel and R\"ockner \cite{roecknen}, one sees that $N=\big(N(t)\big)_{t\geq 0}$ is a strongly continuous convex transition semigroup, where the strong continuity (property (iii) in Definition \ref{def.semigroup}) follows from $T(t)u\leq N(t)u\leq S(t)u$ for all $u\in \BUC$.\\
         Let $u\in D(B)\cap \BUC^1$ and $\ep>0$. It follows from the proof of Lemma~\ref{lem.keygenerator} that, for sufficiently small $h>0$,
         \[
          \phi^*\big(\|u'(x)\|\big)-\ep\leq \frac{\big(E(h)u\big)(x)-\big(T(h)u\big)(x)}{h}.
         \]
         Since $E(t)u\leq N(t)u\leq S(t)u\leq I(t)u$, for $t\geq 0$, it follows that
         \[
          \bigg\|\frac{N(t)u-u}{u}-A u\bigg\|_\infty\to 0\quad \text{as }t\downarrow 0,
         \]
         showing that the generator of $S$ and $N$ coincide on $D(B)\cap \BUC^1$. Similar as in Section~\ref{sec.viscosity}, one sees that the map $t\mapsto N(t)u_0$, for $u_0\in \BUC$, is a viscosity solution to the abstract PDE $$\partial_t u=Au\quad \text{with}\quad u(0)=u_0.$$
         In certain cases, e.g., if $X=\R$ and $B=\frac{\sigma^2}{2}\partial_{xx}$ (see Section \ref{ex.levy}) or if $X=(0,\infty)$ and $B=\frac{\sigma^2 x^2}{2}\partial_{xx}$ (see Section \ref{ex.geombm}), it follows from standard comparison results for viscosity solutions to HJB equations, cf.\ Crandall et al.\ \cite{MR1118699}, that
         \begin{equation}\label{eq.param}
          S(t)u=N(t)u\quad \text{for all }t\geq 0\text{ and }u\in \BUC.
         \end{equation}
          Approximating $u\in \Lip$ with $(u_n)_{n\in \N}$ in $\BUC$ as in Section \ref{sec.unbdd}, it follows that \eqref{eq.param} holds for all $u\in \Lip$. We therefore see that the nonparametric uncertainty captured by $I(t)$ collapses to pure drift uncertainty as $t\downarrow 0$. 
          \begin{remark}
         	We discuss the link to dynamic risk measures. For $u\in \Cb$, the value $\big(I(t)u\big)(x)$ in \eqref{eq.isometry} can be viewed as the conditional risk in terms of a capital requirement 
         	of the position $u(\Xi_t^x)$ under a convex risk measure conditioned on $\{\Xi_0^x=x\}$.\footnote{Instead of $I(t)$, one could even consider the entropic version $\big(\tilde I(t)u\big)(x):=\big(\log I(t)e^u\big)(x)$, where the position $u(\Xi_t^x)$ is weighted with the exponential loss function.}
         	In this context, for $t=k2^{-n}$, we interpret $I^n(t)u=I(2^{-n})^ku$ in \eqref{eq.iteratedyadic} as the $k$-times composition of the conditional risk measures $I(2^{-n})$, which is a time-consistent dynamic risk measure at the discrete time points $\{0,2^{-n},\dots,k2^{-n}\}$. 
         	By making the partition progressively finer, we obtain, in the limit, a dynamic risk measure in continuous time. Here, $\big(S(t)u\big)(x)$ is the dynamic risk of the position $u(\Xi_t^x)$ conditioned on $\{\Xi_0^x=x\}$, and the time-consistency is ensured by the semigroup property. As a result of the Lemma~\ref{lem.monotonicity}, we obtain that finer partitions lead to smaller capital requirements, i.e., $\big(I^{n+1} u\big)(x)\le \big(I^{n} u\big)(x)$ for all $n\in\N$. % In case that the one-step conditional risk measures incorperate the non-parametric Wasserstein uncertainty modeled by $I(2^{-n})$,
         	The previous discussion leads to the following insight: although we start with a family of risk measures $(I^n(t))_{n\in \N}$ that bear nonparametric uncertainty, in the limit, we obtain a dynamic risk measure $S(t)$ that incorprates only parametric drift-uncertainty. In the particular case of a Brownian filtration, the limiting dynamic risk measures are represented as $g$-expectations, see Rosazza Gianin~\cite{MR2241848}.
         	
         	Moreover, due to Lemma~\ref{lem.monotonicity}, we obtain that the limiting dynamic risk measure is dominated by the one-step risk measure, i.e., $S(t)u\le I(t)u$ for all $u\in\Cb$ and $t\ge 0$. In particular, the inequality in \eqref{eq.sensitivity1} allows to simulate bounds for  dynamic risk measures ($g$-expectations) in terms of the reference dynamics and the convex conjugate $\varphi^*$. In the case $X=\R$ with $\psi_t(x)=x$ for all $x\in \R$, $\mu_t=\mathcal{N}(0,t)$, and $\phi(v)=\tfrac{v^2}{2}$, we obtain the entropic risk measure
         	\[
         	\big(S(t)u\big)(x)=\log\int_{\R}\exp\big(u(x+y)\big)\,\mu_t(\d y),
         	\]
         	which is dominated by the respective one-step conditional risk measure $I(t)u$, see also Blessing et al.~\cite[Corollary 5.12]{blessing2022convex} for a discussion on the corresponding Talagrand $T_2$ inequality.
         \end{remark}

	\subsection{Martingale constraints}\label{rem.martingaleconstraint}
	Motivated by risk-neutral pricing, we consider the case where $\int_Xy\, \mu_t(\d y)=0$ for all $t\geq 0$. In line with Bartl et al.~\cite[Section 3.1]{bartl2020robust}, one can add a martingale constraint to the operator $I(t)$, for $t\geq 0$, by considering
	\[
	\big(I_{\rm Mart}(t)u\big)(x):=\sup_{\nu\in \Pc_p^0(X)}\bigg(\int_Xu\big(\psi_t(x)+z\big)\, \nu(\d z)-\phi_t\big(\Wc_p(\mu_t,\nu)\big)\bigg)
	\]
	for all $t\geq 0$, $u\in \BUC$, and $x\in X$, where $\Pc_p^0(X)$ denotes the set of all $\nu\in \Pc_p(X)$ with $\int_Xz\, \nu(\d z)=0$. The proof of Lemma~\ref{lem.keygenerator} shows that, for $u\in \BUC^1$,
	\[
	\frac{I_{\rm Mart}(h)u-T(h)u}{h}\to 0\quad \text{as }h\downarrow 0.
	\]
	In fact, for all $h\in [0,1]$ and $x\in X$, let $\nu_h^x \in \Pc_p^0(X)$ with $\Wc_p(\mu_h,\nu_h^x)\leq bh^\alpha$ and 
	\begin{equation}\label{eq.keygenerator4}
		\big(I_{\rm Mart}(h)u\big)(x)\leq \ep h+\int_X u\big(\psi_h(x)+z\big)\, \nu_h^x(\d z)-\phi_h\big(\Wc_p(\mu_h,\nu_h^x)\big).
	\end{equation}
	For each $h\in [0,1]$ and $x\in X$, let $\pi_h^x\in \Pc_p(X\times X)$ be an optimal coupling between $\mu_h$ and $\nu_h^x$. Then, using the fundamental theorem of infinitesimal calculus and the fact that $\nu_h^x\in \Pc_p^0(X)$, we obtain that
	\begin{align*}
		&\frac{\big(I_{\rm Mart}(h)u\big)(x)-\big(T(h)u\big)(x)}{h}\\
		&\leq  \frac{1}{h}\int_{X\times X}\int_0^1u'\big(\psi_h(x)+(1-s)y+sz\big)(z-y)\, \d s\, {\pi_h^x}(\d y,\d z)-\frac{\phi_h\big(\Wc_p(\mu_h,\nu_h^x)\big)}{h}+\ep\\
		&=  \frac{1}{h}\int_{X\times X}\int_0^1\Big(u'\big(\psi_h(x)+(1-s)y+sz\big)-u'\big(\psi_h(x)\big)\Big)(z-y)\, \d s\, {\pi_h^x}(\d y,\d z)\\
		&\quad\quad-\phi\bigg(e^{-c h}\frac{\Wc_p(\mu_h,\nu_h^x)}{h}\bigg)+\ep\\
		& \leq \bigg(\int_{X\times X}\int_0^1\big\|u'\big(\psi_h(x)+(1-s)y+sz\big)-u'\big(\psi_h(x)\big)\big\|^q \,\d s \, \pi_h^x(\d y,\d z)\bigg)^{1/q}\\
		&\quad\quad\; \times \frac{\Wc_p(\mu_h,\nu_h^x)}{h}-\phi\bigg(e^{-c h}\frac{\Wc_p(\mu_h,\nu_h^x)}{h}\bigg)+\ep\\
		& \leq \phi^*\Bigg(e^{c h}\bigg(\int_{X\times X}\int_0^1\big\|u'\big(\psi_h(x)+(1-s)y+sz\big)-u'\big(\psi_h(x)\big)\big\|^q \,\d s\, \pi_h^x(\d y,\d z)\bigg)^{1/q}\Bigg)\\
		&\quad\quad +\ep
	\end{align*}
	for all $h\in [0,1]$ and $x\in X$, where $q=\frac{p}{p-1}$ is the conjugate exponent of $p$. Proceeding as in the proof of Lemma~\ref{lem.keygenerator}, we find that
	\[
	\sup_{\|x\|\leq r}\int_{X\times X}\int_0^1\big\|u'\big(\psi_h(x)+(1-s)y+sz\big)-u'\big(\psi_h(x)\big)\big\|^q \,\d s\, \pi_h^x(\d y,\d z)\to 0\quad \text{as }h\downarrow 0.
	\]
	In particular,
	\[
	\sup_{\|x\|\leq r}\frac{\big(I_{\rm Mart}(h)u\big)(x)-\big(T(h)u\big)(x)}{h}\to 0\quad \text{as }h\downarrow 0.
	\]
	Similar as in Section \ref{sec.wasserstein}, the family $\big(I_{\rm Mart}(t)\big)_{t\geq 0}$ gives rise to a semigroup $\big(S_{\rm Mart}(t)\big)_{t\geq 0}$. With an analogous argumentation as in the proof of Theorem \ref{thm.semigroup}, one sees that
		\[
		%	I_{\rm Mart}(t)I_{\rm Mart}(s)u\leq I_{\rm Mart}(t+s)u
		\frac{S_{\rm Mart}(h)u-u}{h}\to B u\quad \text{as }h\downarrow 0.
		\]
	Using comparison results as in the previous section, we find that, in certain cases, e.g., if $X=\R$ and $B=\frac{\sigma^2}{2}\partial_{xx}$ or if $X=(0,\infty)$ and $B=\frac{\sigma^2 x^2}{2}\partial_{xx}$,
	\begin{equation}\label{eq.param}
		S_{\rm Mart}(t)u=T(t)u\quad \text{for all }t\geq 0\text{ and }u\in \BUC.
	\end{equation}
	%cf.\ Crandall et al.\ \cite{MR1118699}. 
	In particular, under our assumption that the level of uncertainty is proportional to the time interval, we see that $S_{\rm Mart}$ as a limit $I_{\rm Mart}(\pi)$ over progressively finer time partitions $\pi$ coincides with the transition semigroup $T$ of the reference Markov process.  
	
%	With an analogous argumentation as in the proof of Lemma \ref{lem.monotonicity}, one sees that
%	\[
%	%	I_{\rm Mart}(t)I_{\rm Mart}(s)u\leq I_{\rm Mart}(t+s)u
%	\frac{S_{\rm Mart}(h)u-u}{h}\to B u\quad \text{as }h\downarrow 0
%	\]
%	for all $s,t\geq 0$ and $u\in \BUC^1$, so that $S_{\rm Mart}(t)u=T(t)u$.

	\section{Examples}\label{sec.examples}
	\subsection{Koopman semigroups}\label{sec.koop}
	Let $X=\R^d$ with $d\in \N$. In this example, we consider the case of a deterministic drift that might be susceptible to an uncertain random noise. More precisely, for a fixed Lipschitz continuous function $F\colon X\to X$ with Lipschitz constant $c\geq 0$, we consider the initial value problem
	\begin{align}
		\label{eq.odekoop} \partial_t x(t)&=F\big(x(t)\big)\quad \text{for all }t\in \R,\\
		\notag x(0)&=x\in X.
	\end{align}
	For $x\in X$, we define $\big(\psi_t(x)\big)_{t\geq 0}$ as the unique solution to the above initial value problem. As a reference model, we choose the deterministic dynamics
	\[
	 \Xi_t^x=\psi_t(x)\quad \text{for }t\geq 0\text{ and }x\in X.
	\]
        We assume that a random noise $Z$ with uncertain distribution $\nu\in \Pc_p(X)$ enters in an additive way leading to the uncertain stochastic dynamics
        \[
         \psi_t(x)+Z \quad \text{for }t\geq 0\text{ and }x\in X,
        \]
        where $Z$ can be seen as a generalisation of a known random source, such as a Brownian motion. In the setup of our previous discussion, we thus consider the flow $\mu_t=\delta_0$, for $t\geq 0$, where $\delta_0$ is the Dirac measure with barycenter $0$. Assumption (A3) is therefore trivially satisfied for all $p\geq 1$. The consistency condition \eqref{eq.consistency} is exactly the flow property of solutions to the ODE \eqref{eq.odekoop}. Using Gronwall's lemma,
	\begin{equation}\label{eq.koop1}
		\|\psi_t(x_1)-\psi_t(x_2)\|\leq e^{ct}\|x_1-x_2\|\quad  \text{for all }t\geq 0\text{ and }x_1,x_2\in X,
	\end{equation}
	which shows that Assumption (A1) is satisfied.\ Another application of Gronwall's lemma shows that
	\begin{equation}\label{eq.koop2}
		\|\psi_h(x)-x\|\leq he^{ch} \|F(x)\|\leq he^{ch}\big(\|F(0)\|+c\|x\|\big)
	\end{equation}
	for all $h\geq 0$ and $x\in X$. In particular, for every $r\geq 0$,
	\[
	 \sup_{\|x\|\leq r}\|\psi_h(x)-x\|\leq he^{ch}\big(\|F(0)\|+cr\big)\to 0\quad \text{as }h\downarrow 0.
	\]
	Moreover, the estimate in \eqref{eq.koop2} together with the inverse triangular inequality implies that
	\[
	\|x\|-\|\psi_h(x)\|\leq 	\|\psi_h(x)-x\|\leq he^{ch}\big(\|F(0)\|+c\|x\|\big).
	\]
	Hence, for $h_0>0$ with $ch_0e^{ch_0}<1$,
	\[ \inf_{h \in [0,h_0]} \|\psi_{h}(x)\|\geq \|x\|(1-ch_0e^{ch_0})-h_0e^{ch_0}\|F(0)\|\to \infty\quad\text{as }\|x\|\to\infty,\]
%         From \eqref{eq.koop1} and \eqref{eq.koop2}, it follows that
%	\begin{align*}
%		\|x\|&=\big\|\psi_{-h}\big(\psi_h(x)\big)-\psi_{-h}\big(\psi_h(0)\big)\big\|\leq e^{ch}\|\psi_h(x)-\psi_h(0)\|\\
%		&\leq e^{ch}\|\psi_h(x)\|+he^{2ch}\|F(0)\|,
%	\end{align*}
%	which implies that $\|\psi_h(x)\|\geq e^{-Lh}\|x\|-he^{Lh}\|F(0)\|$ for all $h\geq 0$ and $x\in X$. Hence, for every $ h_0 \geq 0 $, 
%	\[ \inf_{h \in [0,h_0]} \|\psi_{h}(x)\|\to \infty\quad\text{as }\|x\|\to\infty, \]
	which shows that Assumption (A2) is satisfied as well. For $t\geq0$, $x\in X$, and $u\in \BUC$, we have
	\[ \big(I(t)u\big)(x) := \sup_{\nu \in \Pc_{p}(X)}\Bigg(\int_{X}u\big(\psi_{t}(x)+z\big)\,\nu(\d z)-\phi_t\Bigg(\bigg(\int_X \|z\|^p\, \nu(\d z)\bigg)^{1/p}\bigg)\Bigg). \]
	Then, the semigroup $S=\big(S(t)\big)_{t\geq 0}$ can be seen as transition kernels of solutions to the ODE \eqref{eq.odekoop} with an additive robust noise.\ The related Kolmogorov equation is a nonlinear PDE, and given by
	\begin{equation}\label{eq.pdekoop}
	\partial_t u(t,x)=\big(D_xu(t,x)\big)\big(F(x)\big)+\phi^*\big(\|D_xu(t,x)\|\big),
	\end{equation}
	where $D_x$ denotes the Fr\'echet derivative in the space variable. Observe that, in contrast to an additive noise in terms of a Brownian motion, no second-order terms appear in  Equation \eqref{eq.pdekoop}.%\ This phenomenon is due to the linear scaling in the penalisation function.
	
	\subsection{Semiflows arising from linear dynamics}\label{sec.lindyn}
	Let $ X $ be a separable Banach space, $m\in X$, and $G\colon D(G)\subset X\to X$ be the generator of a $C_0$-semigroup $P=\big(P(t)\big)_{t\geq 0}$ on $X$ with compact resolvent.\footnote{Recall that a generator $G$ of a $C_0$-semigroup on $X$ has compact resolvent if $(\lambda-G)^{-1}\colon X\to X$ is a compact operator for some (or equivalently all) $\lambda>\omega$, where $\omega\in \R$ is the growth bound of the related semigroup.} A typical example for such a generator $G$ is given by the Dirichlet Laplacian in $L^q(\Omega)$, where $\Omega\subset \R^d$ is a bounded domain and $1\leq q<\infty$, or any generator of a compact $C_0$-semigroup, cf.\ \cite[Section 2.3]{pazy2012semigroups}. For the Dirichlet Laplacian, the compactness of the resolvent follows from the Rellich-Kondrachov theorem.\ For $t\geq 0$ and $x\in X$, let $$\psi_t(x):=P(t)x$$ be the unique solution to the abstract Cauchy problem
	\[
	 \partial_t x(t)= Gx(t)\quad \text{for }t\geq 0,\quad x(0)=x.
	\]
	Since $P$ is a $C_0$-semigroup on $X$, there exist constants $M\geq 0$ and $\omega\geq 0$ such that $\|P(t)x\|\leq Me^{\omega t}\|x\|$ for all $x\in X$ and $t\geq 0$. By passing to the equivalent norm $$\|x\|_P:=\sup_{t\geq 0}e^{-\omega t}\|P(t)x\|\quad \text{for }x\in X,$$
	we may, w.l.o.g., assume that $M=1$. Since $\psi_t(x_1)-\psi_t(x_2)=P(t)(x_1-x_2)$ for all $x_1,x_2\in X$, it follows that Assumption (A1) is satisfied with $c:=\max\{\omega,0\}$. Choosing $K:=(\lambda-G)^{-1}$ with $\lambda:=1+\omega$, it follows that $K\psi_t(x)=\psi_t(K x)$ for all $x\in X$ and $t\geq 0$, since the resolvent commutes with the semigroup.\ Moreover, the resolvent identity\footnote{Note that $Kx\in D(G)$.} $GKx=\lambda Kx-x$ together with $\|K x\|\leq \frac{1}{\lambda-\omega}=1$ for all $x\in X$ yields
	\[
	 \sup_{\|x\|\leq r}\|\psi_h(Kx)-Kx\|\leq h e^{\omega h}\sup_{\|x\|\leq r}\|GKx\|\leq he^{\omega h}(2+\omega)r\to 0
	\]
        as $h\downarrow 0$ for all $r\geq 0$. We choose $Y_t:=\int_0^t P(s)m\,{\rm d}s$ for $t\geq 0$, so that the assumptions (A1), (A2), and (A3) are satisfied. Then, the nonlinear transition semigroup yields a solution to the nonlinear PDE
	\begin{equation}\label{eq.pdelindyn}
	\partial_t u(t,x)= \big(D_xu(t,x)\big)(Gx+m)+\phi^*\big(\|D_xu(t,x)\|\big).
	\end{equation}

	\subsection{Lévy processes}\label{ex.levy}
	In this example, we consider the case, where $(\Xi_t^x)$ is a L\'evy process taking values in a separable Hilbert space $X$. For a detailed discussion in the finite-dimensional case, we refer to Bartl et al.\ \cite{bartl2020limits}. Throughout this example, we consider $ \psi_{t}(x) = x $ for all $t\geq 0$ and $x\in X$. Since $\psi_t$ is the identity for all $t\geq 0$, the compact operator $K$ can be chosen arbitrarily. Furthermore, let $\mu=(\mu_{t})_{t\geq 0}$ in $\Pc_p(X)$ be a family of infinitely divisible distributions with
	\[
	\lim_{h\downarrow 0}\int_X\|y\|^p\, \mu_h(\d y)=0,
	\]
	so that Assumption (A3) holds. A typical example for $\mu$ is the distribution of a Brownian motion with trace class covariance operator. Since $\psi_t$ is the identity for all $t\geq0$, the Assumptions (A1) and (A2) are trivially satisfied. In this case, for $t\geq 0$, $x\in X$, and $u\in \BUC$, the operator $I(t)$ is given by
	\[ (I(t)u)(x) = \sup_{\nu \in \Pc_{p}(X)} \bigg(\int_{X} u(x+z)\,\nu(\d z) - \phi_{t}\big(\Wc_{p}(\mu_t,\nu)\big)\bigg). \]
	Let $B_\mu$ be the generator of the L\'evy process related to the family $\mu$. Then, the nonlinear semigroup $\big(S(t)\big)_{t\geq 0}$ can be computed by solving the PDE
	\begin{equation}\label{eq.pdelevy1}
	\partial_t u(t,x)=B_\mu u(t,x)+\phi^*\big(\|D_xu(t,x)\|\big).
	\end{equation}
	As an illustration, we consider the case, where $H=\R$, $\mu$ is the law of a Brownian motion (starting in $0$) with standard deviation $\sigma>0$, and $\phi=\infty\cdot \eins_{(a,\infty)}$ with $a\geq 0$. %In the context of Mathematical Finance, this corresponds to a Bachelier model with a Wasserstein perturbation in the law of the underlying Brownian motion.
	In this case, $B_\mu=\frac{\sigma^2}{2}\partial_{xx}$, $\phi^*=a|\cdot |$, and the PDE \eqref{eq.pdelevy1} simplifies to
	\begin{equation}\label{eq.pdelevy2}
	\partial_t u(t,x)=\frac{\sigma^2}{2} \partial_{xx}u(t,x)+a |\partial_xu(t,x)|.
	\end{equation}
	\subsection{Ornstein-Uhlenbeck processes}
	We consider a separable Hilbert space $ X $. Let $G\colon D(G)\subset X\to X$ be the generator of a $C_0$-semigroup $P=\big(P(t)\big)_{t\geq 0}$ on $X$ with compact resolvent, and $m\in X$. We modify the approach in Example \ref{sec.lindyn} by adding a noise in terms of a Brownian motion $(W_t)_{t\geq 0}$ taking values in $X$ with a trace class covariance operator $C\colon X\to X$ on some filtered probability space, satisfying the usual assumptions. For $t\geq 0$ and $x\in X$, let $\psi_t(x):=P(t)x$ be the unique solution to the abstract Cauchy problem
	\[
	 \partial_t x(t)= Gx(t),\quad \text{for }t\geq 0,\quad x(0)=x.
	\]
	Choosing $K:=(\lambda-G)^{-1}$ with $\lambda>0$ sufficiently large, we have seen in Example \ref{sec.lindyn} that the Assumptions (A1) and (A2) are satisfied. Then, we consider the Ornstein-Uhlenbeck process $(\Xi_t^x)$ given by
	\[
	 \Xi_t^x=\psi_t(x)+\int_0^t P(s)m\,{\rm d}s+\int_0^t P(t-s)\, \d W_s,\quad \text{for }t\geq 0\text{ and }x\in X.
	\]
        In this case, $\mu_t$ is the law of the stochastic convolution $\int_0^t P(s)m\,{\rm d}s+ \int_0^t P(t-s)\, \d W_s$. By definition, the process $(\Xi_t^x)$ is a mild solution to the SPDE
        \[
         \d \Xi_t^x=(G\Xi_t^x+m)\,\d t+\d W_t,\quad X_0^x=x.
        \]
        By the Burkholder-Davis-Gundy inequality, Assumption (A3) is satisfied. Here, the operator $ I(t) $ is given by
	\[ \big(I(t)u\big)(x) = \sup_{\nu \in \Pc_{p}(X)}\bigg( \int_{X}u\big(\psi_t(x)+z\big)\,\nu(\d z) - \phi_{t}\big(\Wc_{p}(\mu_t,\nu)\big)\bigg) \]
	for $t\geq0$, $u\in \BUC$, and $x\in X$. The related semigroup $\big(S(t)\big)_{t\geq 0}$ can be computed by solving the semilinear PDE
	\begin{equation}\label{eq.pdeou}
	\partial_tu(t,x)=\big\langle Gx+m, D_xu(t,x)\big\rangle+\frac{1}{2}{\rm tr}\big(CD_{x}^2u(t,x)\big)+\phi^*\big(\|D_xu(t,x)\|\big),
	\end{equation}
	where $D^2_x$ denotes the second Fr\'echet derivative in the space variable.\ We point out that, in comparison to \eqref{eq.pdelindyn}, the resulting PDE incorporates second-order terms. This is due to the consideration of the a priori random noise in terms of a Brownian motion. Roughly speaking, \eqref{eq.pdeou} shows how the two forms of noise, aleatoric and epistemic, enter the equation. The aleatoric (random) noise leads to the term $\frac{1}{2}{\rm tr}\big(CD_{x}^2u(t,x)\big)$ and the epistemic (uncertain) noise leads to the term $\phi^*\big(\|D_xu(t,x)\|\big)$. For a detailed discussion of \eqref{eq.pdeou} in the context of parametric uncertainty, we refer to Nendel and R\"ockner \cite{roecknen}.
	
	\subsection{Geometric Brownian motion}\label{ex.geombm}
	We follow Example \ref{ex.GBM}, and consider the case, where $X=(0,\infty)$, endowed with the multiplication as an ``additive'' operation and the norm $\|x\|:=|\log(x)|$ for all $x\in (0,\infty)$.\ Let $\psi_t(x):=x$, for $x\in (0,\infty)$, and $\mu_t$ be given as the law of the random variable
	\[
	Y_t:=\exp\Big(\big(\alpha-\tfrac{\sigma^2}{2}\big)t +\sigma W_t\Big),
	\]
	where $\alpha\in \R$, $\sigma\geq 0$ and $(W_t)_{t\geq 0}$ is a standard Brownian motion on some probability space. We thus consider dynamics $(\Xi_t^x)$ of the form
	\[
	 \Xi_t^x=xY_t\quad \text{for }t\geq 0\text{ and }x\in (0,\infty).
	\]
         Trivially, the Assumptions (A1) and (A2) are satisfied since $\psi_t$ is the identity for all $t\geq 0$. Assumption (A3) is satisfied since
	\[
	 \E(\|Y_h\|^p)^{1/p}=\E(|\log(Y_h)|^p)^{1/p}\leq h\big|\alpha -\tfrac{\sigma^2}{2}\big|+\sigma \E(|W_h|^p)^{1/p}\to 0\quad \text{as }h\downarrow 0.
	\]
	Here, for $t\geq 0$, $x\in X$, and $u\in \BUC$, the operator $ I(t) $ is given by
	\[ \big(I(t)u\big)(x) = \sup_{\nu \in \Pc_{p}(X)} \bigg(\int_{X}u(xy)\,\nu(\d z) - \phi_{t}\big(\Wc_{p}(\mu_t,\nu)\big)\bigg), \]
	where the considered metric $\Wc_p$ is a logarithmic Wasserstein distance. 

	\bibliographystyle{abbrv}

  % \bibliography{wrps}

\end{document}